\documentclass[reqno,9pt]{amsart}

\usepackage{amsfonts,color}
\usepackage{amssymb,amsmath}
\usepackage[all]{xy}
\usepackage[dvips]{graphics}

\makeatletter
\newcommand{\leqnomode}{\tagsleft@true}
\makeatother

 \newtheorem{theorem}{Theorem}
 \newtheorem{corollary}[theorem]{Corollary}
 \newtheorem{lemma}[theorem]{Lemma}
 \newtheorem{proposition}[theorem]{Proposition}
 \theoremstyle{definition}
 
 \theoremstyle{remark}
 \newtheorem{remark}[theorem]{Remark}

 \numberwithin{equation}{section}
 \numberwithin{theorem}{section}

\begin{document}

\title[$q$-concave operators and $q$-concave Banach lattices]
{Optimal domain of $q$-concave operators and vector measure
representation of $q$-concave Banach lattices}

\author[O.\ Delgado]{O.\ Delgado}
\address{Departamento de Matem\'atica Aplicada I, E.\ T.\ S.\ de Ingenier\'ia
de Edificaci\'on, Universidad de Sevilla, Avenida de Reina Mercedes,
4 A,  Sevilla 41012, Spain}
\email{\textcolor[rgb]{0.00,0.00,0.84}{olvido@us.es}}

\author[E.\ A.\ S\'{a}nchez P\'{e}rez]{E.\ A.\ S\'{a}nchez P\'{e}rez}
\address{Instituto Universitario de Matem\'atica Pura y Aplicada,
Universitat Polit\`ecnica de Val\`encia, Camino de Vera s/n, 46022
Valencia, Spain.}
\email{\textcolor[rgb]{0.00,0.00,0.84}{easancpe@mat.upv.es}}



\thanks{The first author gratefully acknowledges the support of the Ministerio de Econom\'{\i}a y Competitividad
(project \#MTM2012-36732-C03-03) and the Junta de Andaluc\'{\i}a
(projects FQM-262 and FQM-7276), Spain.}
\thanks{The second author acknowledges with thanks the support of the Ministerio de Econom\'{\i}a y Competitividad
(project \#MTM2012-36740-C02-02), Spain.}

\date{\today}

\maketitle


\begin{abstract}
Given a Banach space valued $q$-concave linear operator $T$ defined
on a $\sigma$-order continuous quasi-Banach function space, we
provide a description of the optimal domain of $T$ preserving
$q$-concavity, that is, the largest $\sigma$-order continuous
quasi-Banach function space to which $T$ can be extended as a
$q$-concave operator. We show in this way the existence of maximal extensions for $q$-concave operators. As an application, we show a representation
theorem for $q$-concave Banach lattices through spaces of integrable
functions with respect to a vector measure. This result culminates a series of representation theorems for Banach lattices using vector measures that have been obtained in the last twenty years.
\end{abstract}



\section{Introduction}\label{SEC: Introduction}

Let $X(\mu)$ be a $\sigma$-order continuous quasi-Banach function
space related to a positive measure $\mu$ on a measurable space
$(\Omega,\Sigma)$ such that there exists $g\in X(\mu)$ with $g>0$
$\mu$-a.e.\ and let $T\colon X(\mu)\to E$ be a continuous linear
operator with values in a Banach space $E$. Considering the
$\delta$-ring $\Sigma_{X(\mu)}$ of all sets $A\in\Sigma$ satisfying
that $\chi_A\in X(\mu)$ and the vector measure
$m_T\colon\Sigma_{X(\mu)}\to E$ given by $m_T(A)=T(\chi_A)$, it
follows that the space $L^1(m_T)$ of integrable functions with
respect to $m_T$ is the optimal domain of $T$ preserving continuity.
That is, the largest $\sigma$-order continuous quasi-Banach function
space to which $T$ can be extended as a continuous operator still
with values in $E$. Moreover, the extension of $T$ to $L^1(m_T)$ is
given by the integration operator $I_{m_T}$. This fact was
originally proved in \cite[Corollary 3.3]{curbera-ricker1} for
Banach function spaces $X(\mu)$ with $\mu$ finite and
$\chi_\Omega\in X(\mu)\subset L^1(\mu)$, in which case
$\Sigma_{X(\mu)}$ coincides with the $\sigma$-algebra $\Sigma$. The
extension for Banach function spaces (without extra assumptions) is
deduced from \cite[Proposition 4]{calabuig-delgado-sanchezperez}.
The jump to quasi-Banach function spaces appears in \cite[Theorem
4.14]{okada-ricker-sanchezperez} for the case when $\mu$ is finite
and $\chi_\Omega\in X(\mu)$ and in \cite{delgado-sanchezperez} for
the general case.

Some effort has been made in recent years to solve several versions
of the following general \emph{problem:} Suppose that the operator
$T$ has a property P. Is there an optimal domain for $T$ preserving
P? that is, is there a function space $Z$ such that $T$ can be
extended to $Z$ as an operator with the property P in such a way
that $Z$ is the largest space for which this holds? And in this
case, which is the relation among $Z$ and $m_T$? The answer to the first question is in
general no. For example, in \cite{okada} it is proved that for
compacness or weak compacness $T$ has an optimal domain only in the
case when $I_{m_T}$ is compact or weakly compact, respectively. In
the same line it is shown in
\cite{calabuig-jimenezfernadez-juan-sanchezperez} that $T$ has an
optimal domain for AM-compacness if and only if $I_{m_T}$ is
AM-compact. However, other properties have got  positive answers to
our problem, see \cite{calabuig-jimenezfernadez-juan-sanchezperez}
for narrow operators, \cite{calabuig-delgado-sanchezperez} for
order-w continuous or $Y(\eta)$-extensible operators and
\cite{delgado2} for positive order continuous operators. Also in
\cite{calabuig-jimenezfernadez-juan-sanchezperez} the problem is
studied for Dunford-Pettis operators, but although some partial
results are shown there, the question of the existence of a maximal
extension is still open.

In this paper we analyze this problem for the case of $q$-concave
operators, obtaining a positive answer. Namely, if $T$ is $q$-concave
we show how to compute explicitly the largest quasi-Banach function
space to which $T$ can be extended preserving $q$-concavity
(Corollary \ref{COR: qConcaveOptimalDomain}). Even more, we prove
that this optimal domain is in fact the $q$-concave core of the
space $L^1(m_T)$ and the maximal extension is given by the
integration operator $I_{m_T}$. These results are obtained as a
particular case of the more general Theorem \ref{THM:
(p,q)PowerConcave-Factorization} which gives the optimal domain for
a class of operators (called $(p,q)$-power-concave) which contains
the $q$-concave operators.

As an application we obtain an improvement in some sense of the
Maurey-Rosenthal factorization of $q$-concave operators acting in
$q$-convex Banach function spaces (Corollary \ref{COR:
MaureyRosenthalFactorization}). The reader can find information
about this nowadays classical topic for example in \cite{defant},
\cite{defant-sanchezperez} and the references therein.

In the last section we provide a new representation theorem for
$q$-concave Banach lattices in terms of a vector measure. This type
of representation theorems has its origin in \cite[Theorem
8]{curbera}, where it is proved that every order continuous Banach
lattice $F$ with a weak unit is order isometric to a space
$L^1(\nu)$ of a vector measure $\nu$ defined on a $\sigma$-algebra.
Later in \cite[Proposition
2.4]{fernandez-mayoral-naranjo-saez-sanchezperez} it is shown that
if moreover $F$ is $p$-convex then it is order isometric to $L^p(m)$
for another vector measure $m$. Similar results work for $F$ without
weak unit but in this case the vector measures used in the
representations of $F$ are defined in a $\delta$-ring, see
\cite[Theorem 5]{delgado-juan} and \cite[Theorem
10]{calabuig-juan-sanchezperez}. Also there are representation
theorems for $F$ replacing $\sigma$-order continuity by the Fatou
property, in this case through spaces of weakly integrable
functions, see \cite{curbera-ricker2}, \cite{curbera-ricker3},
\cite{delgado-juan} and \cite{juan-sanchezperez}. For
$p,q\in[1,\infty)$, in Theorem \ref{THM:
BanachLattice-Representation} we obtain that every $q$-concave and
$p$-convex Banach lattice is order isometric to a space $L^p(m)$ of
a vector measure $m$ defined on a $\delta$-ring whose integration
operator $I_{m_T}$ is $\frac{q}{p}$-concave. The converse is also
true. In particular, every $q$-concave Banach lattice is order
isometric to a space $L^1(m)$ of a vector measure $m$ having a
$q$-concave integration operator.


\section{Preliminaries}\label{SEC: Preliminaries}

In this section we establish the notation and present the basic
results on quasi-Banach function spaces (including the proof of some
of them for completeness) and on vector measure integration, which
will be used through the whole paper.

Let $(\Omega,\Sigma)$ be a fixed measurable space. For a measure
$\mu\colon\Sigma\to[0,\infty]$, we denote by $L^0(\mu)$ the space of
all $\Sigma$--measurable real valued functions on $\Omega$, where
functions which are equal $\mu$--a.e.\ are identified.

Given two set functions $\mu,\lambda\colon\Sigma\to[0,\infty]$ we
will write $\lambda\ll\mu$ if $\mu(A)=0$ implies $\lambda(A)=0$. If
$\lambda\ll\mu$ and $\mu\ll\lambda$ we will say that $\mu$ and
$\lambda$ are \emph{equivalent}. If
$\mu,\lambda\colon\Sigma\to[0,\infty]$ are two measures with
$\lambda\ll\mu$, then the map $[i]\colon L^0(\mu)\to L^0(\lambda)$
which takes a $\mu$--a.e.\ class in $L^0(\mu)$ represented by $f$
into the $\lambda$--a.e.\ class represented by the same $f$, is a
well defined linear map. In order to simplify notation we will write
$[i](f)=f$. Note that if $\lambda$ and $\mu$ are equivalent then
$L^0(\mu)=L^0(\lambda)$ and $[i]$ is the identity map $i$.

\subsection{Quasi-Banach function spaces}

Let $X$ be a real vector space and $\Vert\cdot\Vert_X$ a
\emph{quasi-norm} on $X$, that is a function
$\Vert\cdot\Vert_X\colon X\to [0,\infty)$ satisfying the following
conditions:
\begin{itemize}\setlength{\leftskip}{-2.5ex}\setlength{\itemsep}{1ex}
\item[(i)] $\Vert x\Vert_X=0$ if and only if $x=0$,

\item[(ii)] $\Vert \alpha x\Vert_X=\vert\alpha\vert\cdot\Vert x\Vert_X$
for all $\alpha\in\mathbb{R}$ and $x\in X$, and

\item[(iii)] there is a constant $K\ge1$ such that $\Vert
x+y\Vert_X\le K(\Vert x\Vert_X+\Vert y\Vert_X)$ for all $x,y\in X$.
\end{itemize}
For $0<r\le1$ being such that $K=2^{\frac{1}{r}-1}$, it follows that
\begin{equation}\label{EQ: r-sum}
\Big\Vert\sum_{j=1}^nx_j\Big\Vert_X\le
4^{\frac{1}{r}}\Big(\sum_{j=1}^n\Vert
x_j\Vert_X^r\Big)^{\frac{1}{r}}
\end{equation}
for every finite subset $(x_j)_{j=1}^n\subset X$, see \cite[Lemma
1.1]{kalton-peck-roberts}. The quasi-norm $\Vert\cdot\Vert_X$
induces a metrizable vector topology on $X$ where a base of
neighborhoods of $0$ is given by sets of the form $\{x\in X:\, \Vert
x\Vert_X\le \frac{1}{n}\}$. So, a sequence $(x_n)$ converges to $x$
in $X$ if and only if $\Vert x-x_n\Vert_X\to0$. If such topology is
complete then $X$ is said to be a \emph{quasi-Banach space}
(\emph{Banach space} if $K=1$).

Having in mind the inequality \eqref{EQ: r-sum}, standard arguments
show the next result.

\begin{proposition}\label{PROP: quasi-normCompleteness}
The following statements are equivalent:
\begin{itemize}\setlength{\leftskip}{-2.5ex}\setlength{\itemsep}{1ex}
\item[(a)] $X$ is complete.

\item[(b)] For every $0<r'\le r$ ($r$ as in \eqref{EQ: r-sum}) it follows that if
$(x_n)\subset X$ is such that $\sum\Vert x_n\Vert_X^{r'}<\infty$
then $\sum x_n$ converges in $X$.

\item[(c)] There exists $r'>0$ satisfying that if
$(x_n)\subset X$ is such that $\sum\Vert x_n\Vert_X^{r'}<\infty$
then $\sum x_n$ converges in $X$.
\end{itemize}
\end{proposition}

Note that if a series $\sum x_n$ converges in $X$ then
\begin{equation}\label{EQ: norm-sum}
\Big\Vert\sum x_n\Big\Vert_X\le 4^{\frac{1}{r}}K\Big(\sum\Vert
x_n\Vert_X^r\Big)^{\frac{1}{r}},
\end{equation}
where $r$ is as in \eqref{EQ: r-sum}. By using the map $|||\cdot|||$
given in \cite[Theorem 1.2]{kalton-peck-roberts}, it is routine to
check that if $x_n\to x$ in $X$ then
\begin{equation}\label{EQ: Limit-quasi-norm}
4^{-\frac{1}{r}}\limsup\Vert x_n\Vert_X\le\Vert
x\Vert_X\le4^{\frac{1}{r}}\liminf\Vert x_n\Vert_X.
\end{equation}

Also note that a linear map $T\colon X\to Y$ between quasi-Banach
spaces is continuous if and only if there exists a constant $M>0$
such that $\Vert Tx\Vert_Y\le M\Vert x\Vert_X$ for all $x\in X$, see
\cite[p.\,8]{kalton-peck-roberts}.

By a \emph{quasi-Banach function space} (briefly, quasi-B.f.s.)\ we
mean a quasi-Banach space $X(\mu)\subset L^0(\mu)$ satisfying that
if $f\in X(\mu)$ and $g\in L^0(\mu)$ with $|g|\le|f|$ $\mu$--a.e.\
then $g\in X(\mu)$ and $\Vert g\Vert_{X(\mu)}\le\Vert
f\Vert_{X(\mu)}$. If $X(\mu)$ is a Banach space we will refer it as
a \emph{Banach function space} (briefly, B.f.s.). In particular, a
quasi-B.f.s.\ is a quasi-Banach lattice for the $\mu$-a.e.\
pointwise order, in which the convergence in quasi-norm of a
sequence implies the convergence $\mu$-a.e.\ for some subsequence.
Let us prove this important fact.

\begin{proposition}\label{PROP: mu-a.e.ConvergenceSubsequence}
If $f_n\to f$ in a quasi-B.f.s.\ $X(\mu)$, then there exists a
subsequence $f_{n_j}\to f$ $\mu$--a.e.
\end{proposition}

\begin{proof}
Let $r$ be as in \eqref{EQ: r-sum}. We can take a strictly
increasing sequence $(n_j)_{j\ge1}$ such that $\Vert
f-f_{n_j}\Vert_{X(\mu)}\le\frac{1}{2^j}$. For every $m\ge1$, since
$$
\sum_{j\ge m}\Vert f-f_{n_j}\Vert_{X(\mu)}^r\le \sum_{j\ge
m}\frac{1}{2^{jr}}<\infty,
$$
by Proposition \ref{PROP: quasi-normCompleteness} and \eqref{EQ:
norm-sum}, it follows that $g_m=\sum_{j\ge m}|f-f_{n_j}|$ converges
in $X(\mu)$ and $\Vert g_m\Vert_{X(\mu)}\le
4^{\frac{1}{r}}K(\sum_{j\ge m}\frac{1}{2^{jr}})^{\frac{1}{r}}$. Fix
$N\ge 1$ and let $A_j^N=\{\omega\in\Omega: \vert
f(\omega)-f_{n_j}(\omega)\vert>\frac{1}{N}\}$. Since
$$
\chi_{\cap_{m\ge1}\cup_{j\ge m}A_j^N}\le\chi_{\cup_{j\ge
m}A_j^N}\le\sum_{j\ge m}\chi_{A_j^N}\le N\sum_{j\ge m}\vert
f-f_{n_j}\vert=Ng_m,
$$
then
$$
\Vert\chi_{\cap_{m\ge1}\cup_{j\ge m}A_j^N}\Vert_{X(\mu)}\le N\Vert
g_m\Vert_{X(\mu)}\le4^{\frac{1}{r}}NK\Big(\sum_{j\ge
m}\frac{1}{2^{jr}}\Big)^{\frac{1}{r}}.
$$
Taking $m\to\infty$ we have that $\Vert\chi_{\cap_{m\ge1}\cup_{j\ge
m}A_j^N}\Vert_{X(\mu)}=0$ and so $\mu(\cap_{m\ge1}\cup_{j\ge
m}A_j^N)=0$. Then $\mu(\cup_{N\ge1}\cap_{m\ge1}\cup_{j\ge
m}A_j^N)=0$, from which $f_{n_j}\to f$ $\mu$-a.e.
\end{proof}

A quasi-B.f.s.\ $X(\mu)$ is \emph{$\sigma$-order continuous} if for
every $(f_n)\subset X(\mu)$ with $f_n\downarrow0$ $\mu$-a.e.\ it
follows that $\Vert f_n\Vert_X\downarrow0$. It has the
\emph{$\sigma$-Fatou property} if for every sequence $(f_n)\subset
X$ such that $0\le f_n\uparrow f$ $\mu$-a.e.\ and $\sup_n\Vert
f_n\Vert_X<\infty$ we have  that $f\in X$ and $\Vert
f_n\Vert_X\uparrow\Vert f\Vert_X$.

A similar argument to that given in
\cite[p.\,2]{lindenstrauss-tzafriri} for Banach lattices shows that
every positive linear operator between quasi-Banach lattices is
automatically continuous. In particular, all inclusions between
quasi-B.f.s.\ are continuous.

The intersection $X(\mu)\cap Y(\mu)$ and the sum $X(\mu)+Y(\mu)$ of
two quasi-B.f.s.'\ (B.f.s.')\ $X(\mu)$ and $Y(\mu)$ are
quasi-B.f.s.'\ (B.f.s.')\ endowed respectively with the quasi-norms
(norms)
$$
\Vert f\Vert_{X(\mu)\cap Y(\mu)}=\max\big\{\Vert
f\Vert_{X(\mu)},\Vert f\Vert_{Y(\mu)}\big\}
$$
and
$$
\Vert f\Vert_{X(\mu)+Y(\mu)}=\inf\big(\Vert f_1\Vert_{X(\mu)}+\Vert
f_2\Vert_{Y(\mu)}\big),
$$
where the infimum is taken over all possible representations
$f=f_1+f_2$ $\mu$-a.e.\ with $f_1\in X(\mu)$ and $f_2\in Y(\mu)$.
The $\sigma$-order continuity is also preserved by this operations:
if $X(\mu)$ and $Y(\mu)$ are $\sigma$-order continuous then
$X(\mu)\cap Y(\mu)$ and $X(\mu)+Y(\mu)$ are $\sigma$-order
continuous. Detailed proofs of these facts can be found in
\cite{delgado-sanchezperez}, see also \cite[\S\,3, Theorem
1.3]{bennett-sharpley} for the standard parts.

Let $p\in(0,\infty)$. The \emph{$p$-power} of a quasi-B.f.s.\
$X(\mu)$ is the quasi-B.f.s.
$$
X(\mu)^p=\big\{f\in L^0(\mu):\, |f|^p\in X(\mu)\big\}
$$
endowed with the quasi-norm
$$
\Vert
f\Vert_{X(\mu)^p}=\Vert\,|f|^p\,\Vert_{X(\mu)}^{\frac{1}{p}}.
$$
The reader can find a complete explanation of the space $X^p(\mu)$
for instance in \cite[\S\,2.2]{okada-ricker-sanchezperez} for the
case when $\mu$ is finite and $\chi_\Omega\in X(\mu)$. The proofs
given there, with the natural modifications, work in our general
case. However, note that the notation is different: our $p$-powers
here are the $\frac{1}{p}$-th powers there. This standard space can
be found in different sources, unfortunately, notation is not
exactly the same in all of them.

The following remark collects some results on the space $X(\mu)^p$
which will be used in the next sections. First, recall that a
quasi-B.f.s.\ $X(\mu)$ is \emph{$p$-convex} if there exists a
constant $C>0$ such that
$$
\Big\Vert\Big(\sum_{j=1}^n|f_j|^p\Big)^{\frac{1}{p}}\Big\Vert_{X(\mu)}\le
C\,\Big(\sum_{j=1}^n\Vert f_j\Vert_{X(\mu)}^p\Big)^{\frac{1}{p}}
$$
for every finite subset $(f_j)_{j=1}^n\subset X(\mu)$. The smallest
constant satisfying the previous inequality is called the
\emph{p-convexity constant} of $X(\mu)$ and is denoted by
$M^{(p)}[X(\mu)]$.

\begin{remark}\label{REM: XpResults}
Let $X(\mu)$ be a quasi-B.f.s. The following statements hold:
\begin{itemize}\setlength{\leftskip}{-2.5ex}\setlength{\itemsep}{1ex}
\item[(a)] $X(\mu)^p$ is
$\sigma$-order continuous if and only if $X(\mu)$ is $\sigma$-order
continuous.

\item[(b)] If $\chi_\Omega\in X(\mu)$ and $0<p\le q<\infty$ then $X(\mu)^q\subset X(\mu)^p$.

\item[(c)] If $X(\mu)$ is a B.f.s.\ then $X(\mu)^p$ is $p$-convex.

\item[(d)] If $X(\mu)$ is a B.f.s.\ and $p\ge1$ then $\Vert \cdot\Vert_{X(\mu)^p}$ is a norm
and so $X(\mu)^p$ is a B.f.s.

\item[(e)] If $X(\mu)$ is $\frac{1}{p}$-convex with $M^{(\frac{1}{p})}[X(\mu)]=1$ then
$\Vert \cdot\Vert_{X(\mu)^p}$ is a norm
and so $X(\mu)^p$ is a B.f.s.
\end{itemize}
\end{remark}

Let $T\colon X(\mu)\to E$ be a linear operator defined on a
quasi-B.f.s.\ $X(\mu)$ and with values in a quasi-Banach space $E$.
For $q\in(0,\infty)$, the operator $T$ is said to be
\emph{$q$-concave} if there exists a constant $C>0$ such that
$$
\Big(\sum_{j=1}^n\Vert T(f_j)\Vert_E^q\Big)^{\frac{1}{q}}\le
C\,\Big\Vert\Big(\sum_{j=1}^n|f_j|^q\Big)^{\frac{1}{q}}\Big\Vert_{X(\mu)}
$$
for every finite subset $(f_j)_{j=1}^n\subset X(\mu)$. A
quasi-B.f.s.\ $X(\mu)$ is \emph{q-concave} if the identity map
$i\colon X(\mu)\to X(\mu)$ is q-concave.

Note that if $T$ is $q$-concave then it is $p$-concave for all
$p>q$. A proof of this fact can be found in \cite[Proposition
2.54.(iv)]{okada-ricker-sanchezperez} for the case when $\mu$ is
finite and $\chi_\Omega\in X(\mu)$. An adaptation of this proof to
our context works.

\begin{proposition}\label{PROP: q-concaveImpliesSigmaoc}
If $X(\mu)$ is a $q$-concave quasi-B.f.s.\ then it is $\sigma$-order
continuous.
\end{proposition}

\begin{proof}
Since $q$-concavity implies $p$-concavity for every $q<p$, we only
have to consider the case $q\ge1$. Denote by $C$ the $q$-concavity
constant of $X(\mu)$ and consider $(f_n)\subset X(\mu)$ such that
$f_n\downarrow0$ $\mu$-a.e. For every strictly increasing
subsequence $(n_k)$ we have that
\begin{eqnarray*}
\Big(\sum_{k=1}^m\Vert
f_{n_k}-f_{n_{k+1}}\Vert_{X(\mu)}^q\Big)^{\frac{1}{q}} & \le &
C\,\Big\Vert\Big(\sum_{k=1}^m|f_{n_k}-f_{n_{k+1}}|^q\Big)^{\frac{1}{q}}\Big\Vert_{X(\mu)}
\\ & \le & C\,\Big\Vert\sum_{k=1}^m|f_{n_k}-f_{n_{k+1}}|\,\Big\Vert_{X(\mu)} \\ & = &
C\,\Vert f_{n_1}-f_{n_{m+1}}\Vert_{X(\mu)} \\ & \le & C\,\Vert
f_{n_1}\Vert_{X(\mu)}
\end{eqnarray*}
for all $m\ge1$. Then, $(f_n)$ is a Cauchy sequence in $X(\mu)$, as
in other case we can find $\delta>0$ and two subsequences $(n_k)$,
$(m_k)$ such that $n_k<m_k<n_{k+1}<m_{k+1}$ and $\delta<\Vert
f_{n_k}-f_{m_k}\Vert_{X(\mu)}\le\Vert
f_{n_k}-f_{n_{k+1}}\Vert_{X(\mu)}$ for all $k$, which is a
contradiction. Let $h\in X(\mu)$ be such that $f_n\to h$ in
$X(\mu)$. From Proposition \ref{PROP:
mu-a.e.ConvergenceSubsequence}, there exists a subsequence
$f_{n_j}\to h$ $\mu$--a.e.\ and so $h=0$ $\mu$-a.e. Hence, $\Vert
f_n\Vert_{X(\mu)}\downarrow0$.
\end{proof}

\begin{lemma}\label{LEM: q-concaveOperatorOnX+Y}
Let $X(\mu)$ and $Y(\mu)$ be two quasi-B.f.s.'\ and consider a
linear operator $T\colon X(\mu)+Y(\mu)\to E$ with values in a
quasi-Banach space $E$. The operator $T$ is $q$-concave if and only
if both $T\colon X(\mu)\to E$ and $T\colon Y(\mu)\to E$ are
$q$-concave.
\end{lemma}

\begin{proof}
If $T\colon X(\mu)+Y(\mu)\to E$ is $q$-concave, since $X(\mu)\subset
X(\mu)+Y(\mu)$ continuously, it follows that $T\colon X(\mu)\to E$
is $q$-concave. Similarly, $T\colon Y(\mu)\to E$ is $q$-concave.

Suppose that $T\colon X(\mu)\to E$ and $T\colon Y(\mu)\to E$ are
$q$-concave and denote by $C_X$ and $C_Y$ their respective
$q$-concavity constants. Write $K$ for the constant satisfying the
property (iii) of the quasi-norm $\Vert\cdot\Vert_E$. We will use
the inequality:
\begin{equation}\label{EQ: t-inequality}
(a+b)^t\le \max\{1,2^{t-1}\}(a^t+b^t)
\end{equation}
where $0\le a,b<\infty$ and $0<t<\infty$. Let $(f_j)_{j=1}^n\subset
X(\mu)+Y(\mu)$. For
$h=\big(\sum_{j=1}^n|f_j|^q\big)^{\frac{1}{q}}\in X(\mu)+Y(\mu)$,
consider $h_1\in X(\mu)$ and $h_2\in Y(\mu)$ such that $h=h_1+h_2$
$\mu$-a.e. Taking the set $A=\big\{\omega\in\Omega:\, h(\omega)\le
2|h_1(\omega)|\big\}$, $\alpha_q=\max\{1,2^{q-1}\}$ and using
\eqref{EQ: t-inequality}, we have that
\begin{eqnarray*}
\sum_{j=1}^n\Vert T(f_j)\Vert_E^q & \le & K^q\,
\sum_{j=1}^n\Big(\Vert T(f_j\chi_A)\Vert_E+\Vert
T(f_j\chi_{\Omega\backslash A})\Vert_E\Big)^q \\
& \le & K^q\,\alpha_q\,\left(\sum_{j=1}^n\Vert
T(f_j\chi_A)\Vert_E^q+\sum_{j=1}^n\Vert T(f_j\chi_{\Omega\backslash
A})\Vert_E^q\right).
\end{eqnarray*}
Note that $(f_j\chi_A)_{j=1}^n\subset X(\mu)$ as $|f_j|\chi_A\le
h\chi_A\le 2|h_1|$ for all $j$. Then,
\begin{eqnarray*}
\sum_{j=1}^n\Vert T(f_j\chi_A)\Vert_E^q & \le &
C_X^{\,q}\,\Big\Vert\Big(\sum_{j=1}^n|f_j|^q\Big)^{\frac{1}{q}}\chi_A\Big\Vert_{X(\mu)}^q
\\ & = & C_X^{\,q}\,\Vert h\chi_A\Vert_{X(\mu)}^q\le 2^qC_X^{\,q}\,\Vert
h_1\Vert_{X(\mu)}^q.
\end{eqnarray*}
Similarly, $(f_j\chi_{\Omega\backslash A})_{j=1}^n\subset Y(\mu)$ as
$|f_j|\chi_{\Omega\backslash A}\le h\chi_{\Omega\backslash A}\le
2|h_2|$ $\mu$-a.e.\ for all $j$ and so
$$
\sum_{j=1}^n\Vert T(f_j\chi_{\Omega\backslash A})\Vert_E^q \le
2^qC_Y^{\,q}\,\Vert h_2\Vert_{Y(\mu)}^q.
$$
Denoting $C=\max\{C_X,C_Y\}$ and using again \eqref{EQ:
t-inequality}, it follows that
\begin{eqnarray*}
\Big(\sum_{j=1}^n\Vert T(f_j)\Vert_E^q\Big)^{\frac{1}{q}} & \le &
2KC\alpha_q^{\frac{1}{q}}\,\big(\Vert h_1\Vert_{X(\mu)}^q+\Vert
h_2\Vert_{Y(\mu)}^q\big)^{\frac{1}{q}}  \\ & \le &
2^{1+|1-\frac{1}{q}|}KC\,\big(\Vert h_1\Vert_{X(\mu)}+\Vert
h_2\Vert_{Y(\mu)}\big).
\end{eqnarray*}
Taking infimum over all representations
$\big(\sum_{j=1}^n|f_j|^q\big)^{\frac{1}{q}}=h_1+h_2$ $\mu$-a.e.\
with $h_1\in X(\mu)$ and $h_2\in Y(\mu)$, we have that
$$
\Big(\sum_{j=1}^n\Vert
T(f_j)\Vert_E^q\Big)^{\frac{1}{q}}\le2^{1+|1-\frac{1}{q}|}KC\,
\Big\Vert
\Big(\sum_{j=1}^n|f_j|^q\Big)^{\frac{1}{q}}\Big\Vert_{X(\mu)+Y(\mu)}.
$$
\end{proof}

Further information on Banach lattices and function spaces can be
found for instance in \cite{bennett-sharpley,kalton-peck-roberts,lindenstrauss-tzafriri,luxemburg-zaanen,okada-ricker-sanchezperez} and
\cite{zaanen}.

\subsection{Integration with respect to a vector measure defined on a
$\delta$-ring}

Let $\mathcal{R}$ be a \emph{$\delta$--ring} of subsets of $\Omega$
(i.e.\ a ring closed under countable intersections) and let
$\mathcal{R}^{loc}$ be the $\sigma$--algebra of all subsets $A$ of
$\Omega$ such that $A\cap B\in\mathcal{R}$ for all
$B\in\mathcal{R}$. Note that $\mathcal{R}^{loc}=\mathcal{R}$
whenever $\mathcal{R}$ is a $\sigma$-algebra. Write
$\mathcal{S}(\mathcal{R})$ for the space of all
\emph{$\mathcal{R}$--simple functions} (i.e.\ simple functions
supported in $\mathcal{R}$).

A Banach space valued set function $m\colon\mathcal{R}\to E$ is  a
\emph{vector measure} (\emph{real measure} when $E=\mathbb{R}$) if
$\sum m(A_n)$ converges to $m(\cup A_n)$ in $E$ for each sequence
$(A_n)\subset\mathcal{R}$ of pairwise disjoint sets with $\cup
A_n\in\mathcal{R}$.

The \emph{variation} of a real measure $\lambda\colon\mathcal{R}\to
\mathbb{R}$ is the measure
$|\lambda|\colon\mathcal{R}^{loc}\to[0,\infty]$ given by
$$
|\lambda|(A)=\sup\Big\{\sum|\lambda(A_j)|:\, (A_j) \textnormal{
finite disjoint sequence in } \mathcal{R}\cap 2^A\Big\}.
$$
The variation $|\lambda|$ is finite on $\mathcal{R}$. The space
$L^1(\lambda)$ of \emph{integrable functions with respect to
$\lambda$} is defined as the classical space $L^1(|\lambda|)$ with
the usual norm $|f|_\lambda=\int_\Omega|f|\,d|\lambda|$. The
integral of an $\mathcal{R}$--simple function
$\varphi=\sum_{j=1}^na_j\chi_{A_j}$ over $A\in\mathcal{R}^{loc}$ is
defined in the natural way by
$\int_A\varphi\,d\lambda=\sum_{j=1}^na_j\lambda(A_j\cap A)$. The
space $\mathcal{S}(\mathcal{R})$ is dense in $L^1(\lambda)$. This
allows to define the integral of a function $f\in L^1(\lambda)$ over
$A\in\mathcal{R}^{loc}$ as $\int_A
f\,d\lambda=\lim\int_A\varphi_n\,d\lambda$ for any sequence
$(\varphi_n)\subset\mathcal{S}(\mathcal{R})$ converging to $f$ in
$L^1(\lambda)$.

The \emph{semivariation} of a vector measure $m\colon\mathcal{R}\to
E$ is the function $\Vert
m\Vert\colon\mathcal{R}^{loc}\to[0,\infty]$ defined by
$$
\Vert m\Vert(A)=\sup_{x^*\in B_{E^*}}|x^*m|(A),
$$
where $B_{E^*}$ is the closed unit ball of the topological dual
$E^*$ of $E$ and $|x^*m|$ is the variation of the real measure
$x^*m$ given by the composition of $m$ with $x^*$. The semivariation
$\Vert m\Vert$ is finite on $\mathcal{R}$.

A set $A\in\mathcal{R}^{loc}$ is said to be \emph{$m$--null} if
$m(B)=0$ for every $B\in\mathcal{R}\cap 2^A$. This is equivalent to
$\Vert m\Vert(A)=0$. It is known that there exists a measure
$\eta\colon\mathcal{R}^{loc}\to[0,\infty]$ equivalent to $\Vert
m\Vert $ (see \cite[Theorem 3.2]{brooks-dinculeanu}). Denote
$L^0(m)=L^0(\eta)$.

The space $L_w^1(m)$ of \emph{weakly integrable} functions with
respect to $m$ is defined as the space of ($m$-a.e.\ equal)
functions $f\in L^0(m)$ such that $f\in L^1(x^*m)$ for every $x^*\in
E^*$. The space $L^1(m)$ of \emph{integrable} functions with respect
to $m$ consists in all functions $f\in L_w^1(m)$ satisfying that for
each $A\in \mathcal{R}^{loc}$ there exists $x_A \in E$, which is
denoted by $\int_A f\,dm$, such that
$$
x^*(x_A)=\int_A f\,dx^*m,\ \ \textnormal{ for all } x^*\in E^*.
$$
The spaces $L^1(m)$ and $L_w^1(m)$ are B.f.s.'\ related to the
measure space $(\Omega,\mathcal{R}^{loc},\eta)$, and the expression
$$
\Vert f\Vert_m=\sup_{x^*\in B_{E^*}}\int_\Omega |f|\,d|x^*m|
$$
gives a norm for both spaces. The norm of $f\in L^1(m)$ can also be
computed by means of the formula
\begin{equation}\label{EQ: L1m-intnorm}
\Vert f\Vert_m=\sup \left\{ \left\Vert\int_\Omega f\varphi \, d
m\right\Vert_E: \ \varphi \in\mathcal{S}(\mathcal{R}),\,
|\varphi|\le1\right\}.
\end{equation}
Moreover, $L^1(m)$ is $\sigma$-order continuous and contains
$\mathcal{S}(\mathcal{R})$ as a dense subset and $L_w^1(m)$ has the
$\sigma$-Fatou property. For every $\mathcal{R}$-simple function
$\varphi=\sum_{j=1}^n\alpha_j\chi_{A_i}$ it follows that
$\int_A\varphi\,dm=\sum_{j=1}^n\alpha_im(A_j\cap A)$ for all $A\in
\mathcal{R}^{loc}$.

The \emph{integration operator} $I_m\colon L^1(m)\to E$ given by
$I_m(f)=\int_\Omega f\,dm$, is a continuous linear operator with
$\Vert I_m(f)\Vert_E\le\Vert f\Vert_m$. If $m$ is \emph{positive},
that is $m(A)\ge0$ for all $A\in\mathcal{R}$, then $\Vert
f\Vert_m=\Vert I_m(|f|)\Vert_E$ for all $f\in L^1(m)$.

For every $g\in L^1(m)$, the set function
$m_g\colon\mathcal{R}^{loc}\to E$ given by $m_g(A)=I_m(g\chi_A)$ is
a vector measure. Moreover, $f\in L^1(m_g)$ if and only if $fg\in
L^1(m)$, and in this case $\Vert f\Vert_{L^1(m_g)}=\Vert
fg\Vert_{L^1(m)}$.

For definitions and general results regarding integration with
respect to a vector measure defined on a $\delta$-ring we refer to
\cite{calabuig-delgado-juan-sanchezperez,delgado1,lewis,masani-niemi1,masani-niemi2}.

Let $p\in(0,\infty)$. We denote by $L^p(m)$ the $p$-power of
$L^1(m)$, that is,
$$
L^p(m)=\big\{f\in L^0(m):\, |f|^p\in L^1(m)\big\}.
$$
As noted in Remark \ref{REM: XpResults}, the space $L^p(m)$ is a
$\sigma$-order continuous quasi-B.f.s.\ with the quasi-norm $\Vert
f\Vert_{L^p(m)}=\Vert\,|f|^p\,\Vert_{L^1(m)}^{1/p}$. Moreover, if
$p\ge1$ then $\Vert\cdot\Vert_{L^p(m)}$ is a norm and so $L^p(m)$ is
a B.f.s. Direct proofs of these facts and some general results on
the spaces $L^p(m)$ can be found in
\cite{calabuig-juan-sanchezperez}.


\section{The $q$-concave core of a $\sigma$-order continuous
quasi-B.f.s}\label{SEC: q-concaveCore}

Let $X(\mu)$ be a $\sigma$-order continuous quasi-B.f.s.\ and
$q\in(0,\infty)$. We define the space $qX(\mu)$ to be the set of
functions $f\in X(\mu)$ such that
$$
\Vert f\Vert_{qX(\mu)}=\sup\Big(\sum_{j=1}^n\Vert
f_j\Vert_{X(\mu)}^q\Big)^{\frac{1}{q}}<\infty,
$$
where the supremum is taken over all finite set
$(f_j)_{j=1}^n\subset X(\mu)$ satisfying
$|f|=\big(\sum_{j=1}^n|f_j|^q\big)^{\frac{1}{q}}$ $\mu$-a.e. Note
that $\Vert f\Vert_{X(\mu)}\le \Vert f\Vert_{qX(\mu)}$.

\begin{proposition}
The space $qX(\mu)$ is a quasi-B.f.s.\ with quasi-norm $\Vert
\cdot\Vert_{qX(\mu)}$.
\end{proposition}

\begin{proof}
First let us see that if $f\in qX(\mu)$ and $g\in L^0(\mu)$ with
$|g|\le|f|$ $\mu$--a.e.\ then $g\in qX(\mu)$ and $\Vert
g\Vert_{qX(\mu)}\le\Vert f\Vert_{qX(\mu)}$. Note that $g\in X(\mu)$
as $f\in X(\mu)$. Let $(g_j)_{j=1}^n\subset X(\mu)$ be such that
$|g|=\big(\sum_{j=1}^n|g_j|^q\big)^{\frac{1}{q}}$ $\mu$-a.e.\ and
take $h=\big|\,|f|^q-|g|^q\,\big|^{\frac{1}{q}}\in X(\mu)$.  Since
$|f|=\big(\sum_{j=1}^n|g_j|^q+|h|^q\big)^{\frac{1}{q}}$ $\mu$-a.e.,
we have that
$$
\Big(\sum_{j=1}^n\Vert
g_j\Vert_{X(\mu)}^q\Big)^{\frac{1}{q}}\le\Big(\sum_{j=1}^n\Vert
g_j\Vert_{X(\mu)}^q+\Vert
h\Vert_{X(\mu)}^q\Big)^{\frac{1}{q}}\le\Vert f\Vert_{qX(\mu)}.
$$
Taking supremum over all $(g_j)_{j=1}^n\subset X$ with
$|g|=\big(\sum_{j=1}^n|g_j|^q\big)^{\frac{1}{q}}$ $\mu$-a.e., we
have that $g\in qX(\mu)$ with $\Vert g\Vert_{qX(\mu)}\le\Vert
f\Vert_{qX(\mu)}$.

It is direct to check that $\Vert \cdot\Vert_{qX(\mu)}$ satisfies
the properties (i) and (ii) of a quasi-norm. Let $K$ be the constant
satisfying the property (iii) of a quasi-norm for $\Vert
\cdot\Vert_{X(\mu)}$. Given $f,g\in qX(\mu)$ and
$(h_j)_{j=1}^n\subset X$ such that
$|f+g|=\big(\sum_{j=1}^n|h_j|^q\big)^{\frac{1}{q}}$ $\mu$-a.e., by
taking
$A=\big\{\omega\in\Omega:\,|f(\omega)+g(\omega)|\le2|f(\omega)|\big\}$,
$\alpha_q=\max\{1,2^{q-1}\}$ and using \eqref{EQ: t-inequality}, we
have that
\begin{eqnarray*}
\sum_{j=1}^n\Vert h_j\Vert_{X(\mu)}^q & \le &
K^q\sum_{j=1}^n\big(\Vert h_j\chi_A\Vert_{X(\mu)}+\Vert
h_j\chi_{\Omega\backslash A}\Vert_{X(\mu)}\big)^q \\ & \le &
K^q\alpha_q\Big(\sum_{j=1}^n\Vert
h_j\chi_A\Vert_{X(\mu)}^q+\sum_{j=1}^n\Vert
h_j\chi_{\Omega\backslash A}\Vert_{X(\mu)}^q\Big).
\end{eqnarray*}
Note that $|f+g|\chi_A, |f+g|\chi_{\Omega\backslash A}\in qX(\mu)$
as $|f+g|\chi_A\le2|f|$ and $|f+g|\chi_{\Omega\backslash A}\le2|g|$.
Then,
\begin{eqnarray*}
\sum_{j=1}^n\Vert h_j\Vert_{X(\mu)}^q & \le &
K^q\alpha_q\Big(\big\Vert |f+g|\chi_A\big\Vert_{qX(\mu)}^q+\big\Vert
|f+g|\chi_{\Omega\backslash A}\big\Vert_{qX(\mu)}^q\Big)
\\ & \le &
2^qK^q\alpha_q\big(\Vert f\Vert_{qX(\mu)}^q+\Vert
g\Vert_{qX(\mu)}^q\big).
\end{eqnarray*}
By using again \eqref{EQ: t-inequality}, we have that
\begin{eqnarray*}
\Big(\sum_{j=1}^n\Vert h_j\Vert_{X(\mu)}^q\Big)^{\frac{1}{q}} & \le
& 2K\alpha_q^{\frac{1}{q}}\big(\Vert f\Vert_{qX(\mu)}^q+\Vert
g\Vert_{qX(\mu)}^q\big)^{\frac{1}{q}} \\ & \le &
2^{1+|1-\frac{1}{q}|}K\big(\Vert f\Vert_{qX(\mu)}+\Vert
g\Vert_{qX(\mu)}\big).
\end{eqnarray*}
Taking supremum over all $(h_j)_{j=1}^n\subset X$ with
$|f+g|=\big(\sum_{j=1}^n|h_j|^q\big)^{\frac{1}{q}}$ $\mu$-a.e., we
have that
\begin{equation}\label{EQ: pX-Kconstant}
\Vert f+g\Vert_{qX(\mu)}\le2^{1+|1-\frac{1}{q}|}K\big(\Vert
f\Vert_{qX(\mu)}+\Vert g\Vert_{qX(\mu)}\big).
\end{equation}

Finally, let us prove that $qX(\mu)$ is complete. Denote by $r$ and
$r'$ the constants satisfying \eqref{EQ: r-sum} for $X(\mu)$ and
$qX(\mu)$ respectively. Note that $r'<r$ as
$2^{1+|1-\frac{1}{q}|}K>K$. Let $(f_n)\subset qX(\mu)$ be such that
$\sum\Vert f_n\Vert_{qX(\mu)}^{r'}<\infty$. Since
$\Vert\cdot\Vert_{X(\mu)}\le\Vert\cdot\Vert_{qX(\mu)}$, from
Proposition \ref{PROP: quasi-normCompleteness},
 we have that
$\sum_{j=1}^kf_j\to g$ and $\sum_{j=1}^k|f_j|\to \tilde{g}$ in
$X(\mu)$. From Proposition \ref{PROP:
mu-a.e.ConvergenceSubsequence}, it follows that $\sum_{j=1}^kf_j\to
g$ and $\sum_{j=1}^k|f_j|\to \tilde{g}$ pointwise except on a
$\mu$-null set $Z$. Fix any $\gamma>1$ and consider the sets
$A_k=\big\{\omega\in\Omega:\,|g(\omega)|\le
\gamma\sum_{j=1}^k|f_j(\omega)|\big\}$. Note that $\Omega\backslash
\cup A_k\subset Z$ and so it is $\mu$-null. Indeed, if
$\omega\not\in Z$ and $|g(\omega)|>\gamma\sum_{j=1}^k|f_j(\omega)|$
for all $k$ (in particular $\sum|f_n(\omega)|\not=0$), then
$\gamma\sum|f_n(\omega)|\le|g(\omega)|\le\sum|f_n(\omega)|<\infty$,
which is a contradiction. Also note that $g\chi_{A_k}\in qX(\mu)$ as
$|g|\chi_{A_k}\le \gamma\sum_{j=1}^k|f_j|$. Given
$(h_j)_{j=1}^n\subset X$ with
$|g|=\big(\sum_{j=1}^n|h_j|^q\big)^{\frac{1}{q}}$ $\mu$-a.e., we
have that
\begin{eqnarray*}
\Big(\sum_{j=1}^n\Vert
h_j\chi_{A_k}\Vert_{X(\mu)}^q\Big)^{\frac{1}{q}} & \le & \Vert
g\chi_{A_k}\Vert_{qX(\mu)}\le \gamma\Big\Vert
\sum_{j=1}^k|f_j|\Big\Vert_{qX(\mu)} \\ & \le &
4^{\frac{1}{r'}}\gamma\Big(\sum_{j=1}^k\Vert
f_j\Vert_{qX(\mu)}^{r'}\Big)^{\frac{1}{r'}}
\\ & \le &
4^{\frac{1}{r'}}\gamma\Big(\sum\Vert
f_n\Vert_{qX(\mu)}^{r'}\Big)^{\frac{1}{r'}}.
\end{eqnarray*}
On other hand, since $X(\mu)$ is $\sigma$-order continuous and
$|h_j|\chi_{A_k}\uparrow |h_j|$ $\mu$-a.e.\ as $k\to\infty$, we have
that $h_j\chi_{A_k}\to h_j$ in $X(\mu)$ as $k\to\infty$. Taking
limit as $k\to\infty$ in the above inequality and applying
\eqref{EQ: Limit-quasi-norm}, we obtain that
$$
\Big(\sum_{j=1}^n\Vert h_j\Vert_{X(\mu)}^q\Big)^{\frac{1}{q}} \le
4^{\frac{1}{r}+\frac{1}{r'}}\gamma\Big(\sum\Vert
f_n\Vert_{qX(\mu)}^{r'}\Big)^{\frac{1}{r'}}.
$$
Now, taking supremum over all $(h_j)_{j=1}^n\subset X$ with
$|g|=\big(\sum_{j=1}^n|h_j|^q\big)^{\frac{1}{q}}$ $\mu$-a.e., it
follows that $g\in qX(\mu)$ with $\Vert
g\Vert_{qX(\mu)}\le4^{\frac{1}{r}+\frac{1}{r'}}\gamma\big(\sum\Vert
f_n\Vert_{qX(\mu)}^{r'}\big)^{\frac{1}{r'}}$. Even more, since
$\gamma$ is arbitrary, taking $\gamma\to 1$ we have that
$$
\Big\Vert \sum
f_n\Big\Vert_{qX(\mu)}\le4^{\frac{1}{r}+\frac{1}{r'}}\Big(\sum\Vert
f_n\Vert_{qX(\mu)}^{r'}\Big)^{\frac{1}{r'}}.
$$
Of course $\sum_{j=1}^nf_j\to g$ in $qX(\mu)$ as
$$
\Big\Vert g-\sum_{j=1}^nf_j\Big\Vert_{qX(\mu)}=\Big\Vert
\sum_{j>n}f_j\Big\Vert_{qX(\mu)}\le4^{\frac{1}{r}+\frac{1}{r'}}\Big(\sum_{j>n}\Vert
f_j\Vert_{qX(\mu)}^{r'}\Big)^{\frac{1}{r'}}\to0.
$$
Therefore, from Proposition \ref{PROP: quasi-normCompleteness} it
follows that $qX(\mu)$ is complete.
\end{proof}

\begin{proposition}
The space $qX(\mu)$ is $q$-concave. In consequence, it is also $\sigma$-order continuous.
\end{proposition}

\begin{proof}
Let $(f_j)_{j=1}^n\subset qX(\mu)$ and consider
$(h_k^j)_{k=1}^{m_j}\subset X(\mu)$ with
$|f_j|=\big(\sum_{k=1}^{m_j}|h_k^j|^q\big)^{\frac{1}{q}}$
$\mu$-a.e.\ for each $j$. Since
$\big(\sum_{j=1}^n|f_j|^q\big)^{\frac{1}{q}}=\big(\sum_{j=1}^n\sum_{k=1}^{m_j}|h_k^j|^q\big)^{\frac{1}{q}}$
$\mu$-a.e., it follows that
$$
\sum_{j=1}^n\sum_{k=1}^{m_j}\Vert h_k^j\Vert_{X(\mu)}^q\le
\Big\Vert\Big(\sum_{j=1}^n|f_j|^q\Big)^{\frac{1}{q}}\Big\Vert_{qX(\mu)}^q.
$$
Taking supremum for each $j=1,...,n$ over all
$(h_k^j)_{k=1}^{m_j}\subset X(\mu)$ with
$|f_j|=\big(\sum_{k=1}^{m_j}|h_k^j|^q\big)^{\frac{1}{q}}$
$\mu$-a.e., we have that
$$
\sum_{j=1}^n\Vert f_j\Vert_{qX(\mu)}^q\le
\Big\Vert\Big(\sum_{j=1}^n|f_j|^q\Big)^{\frac{1}{q}}\Big\Vert_{qX(\mu)}^q
$$
and so $qX(\mu)$ is $q$-concave. The $\sigma$-order continuity is
given by Proposition \ref{PROP: q-concaveImpliesSigmaoc}.
\end{proof}

Even more, the following proposition shows that $qX(\mu)$ is in fact
the \emph{$q$-concave core} of $X(\mu)$, that is, the largest
$q$-concave quasi-B.f.s.\ related to $\mu$ contained in $X(\mu)$. In
particular, $qX(\mu)=X(\mu)$ whenever $X(\mu)$ is $q$-concave.

\begin{proposition}\label{PROP: q-concaveCore}
Let $Z(\xi)$ be a quasi-B.f.s.\ with $\mu\ll\xi$. The following
statements are equivalent:
\begin{itemize}\setlength{\leftskip}{-3ex}\setlength{\itemsep}{.5ex}
\item[(a)] $[i]\colon Z(\xi)\to X(\mu)$ is
well defined and $q$-concave.

\item[(b)] $[i]\colon Z(\xi)\to qX(\mu)$ is well defined.
\end{itemize}
In particular, $qX(\mu)$ is the $q$-concave core of $X(\mu)$.
\end{proposition}

\begin{proof}
(a) $\Rightarrow$ (b) Denote by $C$ the $q$-concavity constant of
the operator $[i]\colon Z(\xi)\to X(\mu)$. Let $f\in Z(\xi)$ (so
$f\in X(\mu)$) and $(f_j)_{j=1}^{n}\subset X(\mu)$ with
$|f|=\big(\sum_{j=1}^n|f_j|^q\big)^{\frac{1}{q}}$ except on a
$\mu$-null set $N$. Since $|f_j|\chi_{\Omega\backslash N}\le |f|$
pointwise (so $\xi$-a.e.), then $f_j\chi_{\Omega\backslash N}\in
Z(\xi)$. Noting that $f_j=f_j\chi_{\Omega\backslash N}$ $\mu$-a.e.,
it follows that
\begin{eqnarray*}
\Big(\sum_{j=1}^n\Vert f_j\Vert_{X(\mu)}^q\Big)^{\frac{1}{q}} & = &
\Big(\sum_{j=1}^n\Vert f_j\chi_{\Omega\backslash
N}\Vert_{X(\mu)}^q\Big)^{\frac{1}{q}} \\ & \le & C\,\Big\Vert
\Big(\sum_{j=1}^n|f_j|^q\Big)^{\frac{1}{q}}\chi_{\Omega\backslash
N}\Big\Vert_{Z(\xi)}\le C\,\Vert f\Vert_{Z(\xi)}.
\end{eqnarray*}
Hence $f\in qX(\mu)$ with $\Vert f\Vert_{qX(\mu)}\le C\,\Vert
f\Vert_{Z(\xi)}$.

(b) $\Rightarrow$ (a) Clearly $[i]\colon Z(\xi)\to X(\mu)$ is well
defined as $qX(\mu)\subset X(\mu)$. Denote by $M$ the continuity
constant of $[i]\colon Z(\xi)\to qX(\mu)$ (recall that every
positive operator between quasi-B.f.s.'\ is continuous). For every
$(f_j)_{j=1}^n\subset Z(\xi)$ we have that
$\big(\sum_{j=1}^n|f_j|^q\big)^{\frac{1}{q}}$ is in $qX(\mu)$ as it
is in $Z(\xi)$, and so
\begin{eqnarray*}
\Big(\sum_{j=1}^n\Vert f_j\Vert_{X(\mu)}^q\Big)^{\frac{1}{q}} & \le
&\Big\Vert
\Big(\sum_{j=1}^n|f_j|^q\Big)^{\frac{1}{q}}\Big\Vert_{qX(\mu)} \\ &
\le & M\,\Big\Vert
\Big(\sum_{j=1}^n|f_j|^q\Big)^{\frac{1}{q}}\Big\Vert_{Z(\xi)}.
\end{eqnarray*}
Hence, $[i]\colon Z(\xi)\to X(\mu)$ is $q$-concave.

In particular, if $Z(\mu)$ is a $q$-concave quasi-B.f.s.\ such that
$Z(\mu)\subset X(\mu)$, we have that $i\colon Z(\mu)\to X(\mu)$ is
well defined, continuous and so $q$-concave. Then, from (a)
$\Rightarrow$ (b) we have that $Z(\mu)\subset qX(\mu)$.
\end{proof}

For $p\in(0,\infty)$, the $p$-power of $qX(\mu)$ can be described in
terms of the $p$-power of $X(\mu)$.

\begin{proposition}\label{PROP: pPower-qX(mu)}
The equality $\big(qX(\mu)\big)^p=qpX(\mu)^p$ holds with equal
norms.
\end{proposition}

\begin{proof}
Let $f\in \big(qX(\mu)\big)^p$. Since $|f|^p\in qX(\mu)$, in
particular $|f|^p\in X(\mu)$ and so $f\in X(\mu)^p$. Consider
$(f_j)_{j=1}^{n}\subset X(\mu)^p$ satisfying that
$|f|=\big(\sum_{j=1}^n|f_j|^{qp}\big)^{\frac{1}{qp}}$ $\mu$-a.e.
Noting that $(|f_j|^p)_{j=1}^{n}\subset X(\mu)$ and
$|f|^p=\big(\sum_{j=1}^n(|f_j|^p)^q\big)^{\frac{1}{q}}$ $\mu$-a.e.,
we have that
$$
\Big(\sum_{j=1}^n\Vert f_j\Vert_{X(\mu)^p}^{qp}\Big)^{\frac{1}{qp}}=
\Big(\sum_{j=1}^n\Vert\,
|f_j|^p\,\Vert_{X(\mu)}^q\Big)^{\frac{1}{qp}}\le\Vert\,|f|^p\,\Vert_{qX(\mu)}^{\frac{1}{p}}
=\Vert f\Vert_{(qX(\mu))^p}.
$$
Then, $f\in qpX(\mu)^p$ and $\Vert f\Vert_{qpX(\mu)^p}\le\Vert
f\Vert_{(qX(\mu))^p}$.

Let now $f\in qpX(\mu)^p$. In particular $f\in X(\mu)^p$ and so
$|f|^p\in X(\mu)$. Consider $(f_j)_{j=1}^{n}\subset X(\mu)$
satisfying that $|f|^p=\big(\sum_{j=1}^n|f_j|^q\big)^{\frac{1}{q}}$
$\mu$-a.e. Noting that $(|f_j|^{\frac{1}{p}})_{j=1}^{n}\subset
X(\mu)^p$ and
$|f|=\big(\sum_{j=1}^n(|f_j|^{\frac{1}{p}})^{qp}\big)^{\frac{1}{qp}}$
$\mu$-a.e., we have that
$$
\Big(\sum_{j=1}^n\Vert f_j\Vert_{X(\mu)}^q\Big)^{\frac{1}{q}}=
\Big(\sum_{j=1}^n\Vert\,
|f_j|^{\frac{1}{p}}\,\Vert_{X(\mu)^p}^{qp}\Big)^{\frac{1}{q}}\le\Vert
f\Vert_{qpX(\mu)^p}^p.
$$
Then, $|f|^p\in qX(\mu)$ and $\Vert\,|f|^p\,\Vert_{qX(\mu)}\le\Vert
f\Vert_{qpX(\mu)^p}^p$. Hence, $f\in \big(qX(\mu)\big)^p$ and $\Vert
f\Vert_{(qX(\mu))^p}=\Vert\,|f|^p\,\Vert_{qX(\mu)}^{\frac{1}{p}}\le\Vert
f\Vert_{qpX(\mu)^p}$.
\end{proof}


\section{Optimal domain for $(p,q)$-power-concave operators}
\label{SEC: (p,q)PowerConcave-quasiBfs}

Let $X(\mu)$ be a $\sigma$-order continuous quasi-B.f.s.\ satisfying
what we call the \emph{$\sigma$-property}:
$$
\Omega=\cup \Omega_n \textnormal{ with } \chi_{\Omega_n}\in X(\mu)
\textnormal{ for all } n,
$$
and let $T\colon X(\mu)\to E$ be a continuous linear operator with
values in a Banach space $E$. We consider the $\delta$-ring
$$
\Sigma_{X(\mu)}=\big\{A\in\Sigma:\, \chi_A\in X(\mu)\big\}
$$
and the vector measure $m_T\colon\Sigma_{X(\mu)}\to E$ given by
$m_T(A)=T(\chi_A)$. Note that the $\sigma$-property of $X(\mu)$
guarantees that $\Sigma_{X(\mu)}^{loc}=\Sigma$ and since $\Vert
m_T\Vert\ll\mu$ we have that $[i]\colon L^0(\mu)\to L^0(m_T)$ is
well defined. Also note that a quasi-B.f.s.\ has the
$\sigma$-property if and only if it contains a function $g>0$
$\mu$-a.e.

As an extension of \cite[\S\,3]{calabuig-delgado-sanchezperez} to
quasi-B.f.s.', in \cite{delgado-sanchezperez} it is proved that
$[i]\colon X(\mu)\to L^1(m_T)$ is well defined and
$T=I_{m_T}\circ[i]$. Even more, $L^1(m_T)$ is the largest
$\sigma$-order continuous quasi-B.f.s.\ with this property. That
means, if $Z(\xi)$ is a $\sigma$-order continuous quasi-B.f.s.\ with
$\xi\ll\mu$ and $T$ factors as
$$
\xymatrix{
X(\mu) \ar[rr]^{T} \ar@{.>}[dr]_(.45){[i]} & & E\\
& Z(\xi) \ar@{.>}[ur]_(.5){S}}
$$
with $S$ being a continuous linear operator, then $[i]\colon
Z(\xi)\to L^1(m_T)$ is well defined and $S=I_{m_T}\circ[i]$. In
other words, $L^1(m_T)$ is the optimal domain to which $T$ can be
extended preserving continuity.

In this section we present the main results of the paper, including
a description of the optimal domain for $T$ (when $T$ is
$q$-concave) preserving $q$-concavity. First, we have to provide a
natural non-finite measure version of the so called $p$-th power
factorable operators, which were developed for the first time in
\cite[\S\,5.1]{okada-ricker-sanchezperez} for the case of finite
measures. For $p\in(0,\infty)$, we say that $T$ is a \emph{$p$-th
power factorable operator with a continuous extension} if there is a
continuous linear extension of $T$ to $X(\mu)^{\frac{1}{p}}+X(\mu)$,
i.e. $T$ factors as
$$
\xymatrix{
X(\mu) \ar[rr]^{T} \ar@{.>}[dr]_(.4){i} & & E\\
& X(\mu)^{\frac{1}{p}}+X(\mu) \ar@{.>}[ur]_(.6){S}}
$$
for a continuous linear operator $S$.

Regarding this definition and having in mind Remark \ref{REM:
XpResults}.(b), two standard cases must be considered whenever
$\chi_\Omega\in X(\mu)$. If $1<p$ we have that
$X(\mu)^{\frac{1}{p}}+X(\mu)=X(\mu)^{\frac{1}{p}}$, and then the
definition of $p$-th power factorable operator with a continuous
extension coincides with the one given in \cite[Definition
5.1]{okada-ricker-sanchezperez}. However, if $p\le1$ we have that
$X(\mu)^{\frac{1}{p}}+X(\mu)=X(\mu)$ and so $p$-th power factorable
operators with continuous extensions are just continuous operators.

The following result, which is proved in \cite{delgado-sanchezperez}
in order to find the optimal domain for $p$-th power factorable
operators, will be the starting point of our work in this section.
The proof is an adaptation to our setting of the proof given in
\cite[Theorem 5.7]{okada-ricker-sanchezperez} for the case when
$\mu$ is finite, $\chi_\Omega\in X(\mu)$ and $p\ge1$.

\begin{theorem}\label{THM: quasiBfsXsubset(LpCapL1)mT}
The following statements are equivalent.
\begin{itemize}\setlength{\leftskip}{-3ex}\setlength{\itemsep}{1ex}
\item[(a)] $T$ is $p$-th power factorable with a continuous extension.

\item[(b)] $[i]\colon X(\mu)^{\frac{1}{p}}+X(\mu)\to L^1(m_T)$ is well defined.

\item[(c)] $[i]\colon X(\mu)\to L^p(m_T)\cap L^1(m_T)$ is well defined.

\item[(d)] There exists $M>0$ such that $\Vert
Tf\Vert_E\le M\Vert f\Vert_{X(\mu)^{\frac{1}{p}}+X(\mu)}$ for all
$f\in X(\mu)$.
\end{itemize}
Moreover, if (a)-(d) holds, the extension of $T$ to
$X(\mu)^{\frac{1}{p}}+X(\mu)$ coincides with the integration
operator $I_{m_T}\circ[i]$.
\end{theorem}

In a brief overview, (a) implies (b) and the fact that the extension
of $T$ to $X(\mu)^{\frac{1}{p}}+X(\mu)$ is just $I_{m_T}\circ[i]$
follow from the optimality of $L^1(m_T)$. Note that
$X(\mu)^{\frac{1}{p}}+X(\mu)$ is $\sigma$-order continuous as
$X(\mu)$ is so. The equivalence between (b) and (c) is a direct
check. Statement (b) implies (d) since $[i]\colon
X(\mu)^{\frac{1}{p}}+X(\mu)\to L^1(m_T)$ is continuous (as it is
positive) and $T=I_{m_T}\circ[i]$. Finally, (d) implies (a) is based
on a standard argument which use the approximation of a measurable
function through functions in $X(\mu)$ (possible by the
$\sigma$-property) to construct an extension of $T$ to
$X(\mu)^{\frac{1}{p}}+X(\mu)$. For a detailed proof of Theorem
\ref{THM: quasiBfsXsubset(LpCapL1)mT} see
\cite{delgado-sanchezperez}, where moreover it is proved that if $T$
is $p$-th power factorable with a continuous extension then
$L^p(m_T)\cap L^1(m_T)$ is the optimal domain to which $T$ can be
extended preserving this property.

Now, let us go to the new results on optimal domains. We consider
the following property stronger than $p$-th power factorable and
look for its optimal domain.

For $p,q\in(0,\infty)$, we say that $T$ is
\emph{$(p,q)$-power-concave} if there exists a constant $C>0$ such
that
$$
\Big(\sum_{j=1}^n\Vert
T(f_j)\Vert_E^{\frac{q}{p}}\Big)^{\frac{p}{q}}\le
C\,\Big\Vert\Big(\sum_{j=1}^n|f_j|^{\frac{q}{p}}\Big)^{\frac{p}{q}}\Big\Vert_{X(\mu)^{\frac{1}{p}}+X(\mu)}
$$
for every finite subset $(f_j)_{j=1}^n\subset X(\mu)$. If
$\chi_\Omega\in X(\mu)$ and $p\ge1$ we have that
$X(\mu)^{\frac{1}{p}}+X(\mu)=X(\mu)^{\frac{1}{p}}$, and then our
definition of $(p,q)$-power-concave operator coincides with the one
given in \cite[Definition 6.1]{okada-ricker-sanchezperez}.

\begin{remark}\label{REM: (p,q)PowerConcave}
The following statements hold:
\begin{itemize}\setlength{\leftskip}{-2.5ex}\setlength{\itemsep}{1ex}
\item[(i)] A $(1,q)$-power-concave operator is just a $q$-concave operator.

\item[(ii)] If $T$ is $(p,q)$-power-concave then $T$ is
$\frac{q}{p}$-concave, as $X(\mu)\subset
X(\mu)^{\frac{1}{p}}+X(\mu)$ continuously.

\item[(iii)] If $\chi_\Omega\in X(\mu)$ and $p<1$, since $X(\mu)^{\frac{1}{p}}+X(\mu)=X(\mu)$,
we have that $(p,q)$-power-concavity coincides with
$\frac{q}{p}$-concavity.

\item[(iv)] If $T$ is
$(p,q)$-power-concave then $T$ is $p$-th power factorable with a
continuous extension. Indeed, the $(p,q)$-power-concave inequality
applied to an unique function is just the item (d) of Theorem
\ref{THM: quasiBfsXsubset(LpCapL1)mT}
\end{itemize}
\end{remark}

As we will see in the next result, $(p,q)$-power-concavity is close
related to the following property. We say that $T$ is \emph{$p$-th
power factorable with a $q$-concave extension} if there exists a
$q$-concave linear extension of $T$ to
$X(\mu)^{\frac{1}{p}}+X(\mu)$, i.e. $T$ factors as
$$
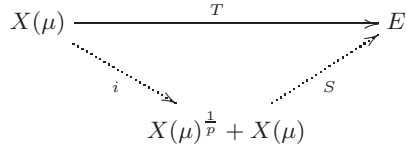
\xymatrix{
X(\mu) \ar[rr]^{T} \ar@{.>}[dr]_(.45){i} & & E\\
& X(\mu)^{\frac{1}{p}}+X(\mu) \ar@{.>}[ur]_(.55){S}}
$$
with $S$ being a $q$-concave linear operator. In this case, it is
direct to check that $T$ is $q$-concave.

\begin{theorem}\label{THM: Xsubset(q/pL1CapqLp)mT}
The following statements are equivalent:
\begin{itemize}\setlength{\leftskip}{-2.5ex}\setlength{\itemsep}{1ex}
\item[(a)] $T$ is $(p,q)$-power-concave.

\item[(b)] $T$ is $p$-th power factorable with a $\frac{q}{p}$-concave extension.

\item[(c)] $[i]\colon X(\mu)^{\frac{1}{p}}+X(\mu)\to L^1(m_T)$ is well defined and
$\frac{q}{p}$-concave.

\item[(d)] $[i]\colon X(\mu)\to L^1(m_T)$ is well defined and $\frac{q}{p}$-concave, and
$[i]\colon X(\mu)\to L^p(m_T)$ is well defined and $q$-concave.

\item[(e)] $[i]\colon X(\mu)\to \frac{q}{p}L^1(m_T)\cap qL^p(m_T)$ is well
defined.
\end{itemize}
Moreover, if (a)-(e) holds, the extension of $T$ to
$X(\mu)^{\frac{1}{p}}+X(\mu)$ coincides with the integration
operator $I_{m_T}\circ[i]$.
\end{theorem}

\begin{proof}
First note that $\frac{q}{p}L^1(m_T)\cap qL^p(m_T)$ is
$\sigma$-order continuous as a consequence of Proposition \ref{PROP:
q-concaveImpliesSigmaoc}.

(a) $\Rightarrow$ (b) From Remark \ref{REM: (p,q)PowerConcave}.(iv)
we have that $T$ is $p$-th power factorable with a continuous
extension. Let $S\colon X(\mu)^{\frac{1}{p}}+X(\mu)\to E$ be a
continuous linear operator extending $T$. We are going to see that
$S$ is $\frac{q}{p}$-concave. Since $T$ is $(p,q)$-power-concave and
$S=T$ on $X(\mu)$, there exists $C>0$ such that
$$
\Big(\sum_{j=1}^n\Vert
S(f_j)\Vert_E^{\frac{q}{p}}\Big)^{\frac{p}{q}}\le
C\,\Big\Vert\Big(\sum_{j=1}^n|f_j|^{\frac{q}{p}}\Big)^{\frac{p}{q}}\Big\Vert_{X(\mu)^{\frac{1}{p}}+X(\mu)}
$$
for all finite subset $(f_j)_{j=1}^n\subset X(\mu)$. Consider
$(f_j)_{j=1}^n\subset X(\mu)^{\frac{1}{p}}+X(\mu)$ with $f_j\ge0$
$\mu$-a.e.\ for all $j$. The $\sigma$-property of $X(\mu)$ allows to
find for each $j=1,...,n$ a sequence $(h_k^j)\subset X(\mu)$ such
that $0\le h_k^j\uparrow f_j$ $\mu$-a.e.\ as $k\to\infty$ (see
\cite{delgado-sanchezperez} for the details). For every $k$, we have
that
\begin{eqnarray*}
\Big(\sum_{j=1}^n\Vert
S(h_k^j)\Vert_E^{\frac{q}{p}}\Big)^{\frac{p}{q}} & \le &
C\,\Big\Vert\Big(\sum_{j=1}^n|h_k^j|^{\frac{q}{p}}\Big)^{\frac{p}{q}}\Big\Vert_{X(\mu)^{\frac{1}{p}}+X(\mu)}
\\ & \le &
C\,\Big\Vert\Big(\sum_{j=1}^n|f_j|^{\frac{q}{p}}\Big)^{\frac{p}{q}}\Big\Vert_{X(\mu)^{\frac{1}{p}}+X(\mu)}.
\end{eqnarray*}
On other hand, since $X(\mu)^{\frac{1}{p}}+X(\mu)$ is $\sigma$-order
continuous, it follows that $h_k^j\to f_j$ in
$X(\mu)^{\frac{1}{p}}+X(\mu)$ as $k\to\infty$, and so $S(h_k^j)\to
S(f_j)$ in $E$ as $k\to\infty$. Hence, taking limit as $k\to\infty$
in the above inequality, it follows that
$$
\Big(\sum_{j=1}^n\Vert
S(f_j)\Vert_E^{\frac{q}{p}}\Big)^{\frac{p}{q}}\le
C\,\Big\Vert\Big(\sum_{j=1}^n|f_j|^{\frac{q}{p}}\Big)^{\frac{p}{q}}\Big\Vert_{X(\mu)^{\frac{1}{p}}+X(\mu)}.
$$
For a general $(f_j)_{j=1}^n\subset X(\mu)^{\frac{1}{p}}+X(\mu)$,
write $f_j=f_j^+-f_j^-$ where $f_j^+$ and $f_j^-$ are the positive
and negative parts respectively of each $f_j$. By using inequality
\eqref{EQ: t-inequality} and denoting
$\alpha_{p,q}=\max\{1,2^{1-\frac{p}{q}}\}$, we have that
\begin{eqnarray*}
\Big(\sum_{j=1}^n\Vert
S(f_j)\Vert_E^{\frac{q}{p}}\Big)^{\frac{p}{q}} & \le &
\Big(\sum_{j=1}^n\big(\Vert S(f_j^+)\Vert_E+\Vert
S(f_j^-)\Vert_E\big)^{\frac{q}{p}}\Big)^{\frac{p}{q}} \\ & \le &
\alpha_{p,q}\Big(\sum_{j=1}^n\Vert
S(f_j^+)\Vert_E^{\frac{q}{p}}+\sum_{j=1}^n\Vert S(f_j^-)\Vert_E^{\frac{q}{p}}\Big)^{\frac{p}{q}} \\
& \le &
\alpha_{p,q}\,C\,\Big\Vert\Big(\sum_{j=1}^n|f_j^+|^{\frac{q}{p}}
+\sum_{j=1}^n|f_j^-|^{\frac{q}{p}}\Big)^{\frac{p}{q}}\Big\Vert_{X(\mu)^{\frac{1}{p}}+X(\mu)}
\\ & = & \alpha_{p,q}\,C\,
\Big\Vert\Big(\sum_{j=1}^n|f_j|^{\frac{q}{p}}\Big)^{\frac{p}{q}}\Big\Vert_{X(\mu)^{\frac{1}{p}}+X(\mu)}
\end{eqnarray*}
(for the last equality note that
$|f_j|^{\frac{q}{p}}=|f_j^+|^{\frac{q}{p}} +|f_j^-|^{\frac{q}{p}}$
as $f_j^+$ and $f_j^-$ have disjoint support).

(b) $\Rightarrow$ (c) Since $T$ is $p$-th power factorable with a
$\frac{q}{p}$-concave (and so continuous) extension, from Theorem
\ref{THM: quasiBfsXsubset(LpCapL1)mT}, the map $[i]\colon
X(\mu)^{\frac{1}{p}}+X(\mu)\to L^1(m_T)$ is well defined. Let
$S\colon X(\mu)^{\frac{1}{p}}+X(\mu)\to E$ be a
$\frac{q}{p}$-concave linear operator extending $T$. Note that
$S=I_{m_T}\circ[i]$ (Theorem \ref{THM: quasiBfsXsubset(LpCapL1)mT}).
Denote by $C$ the $\frac{q}{p}$-concavity constant of $S$. Consider
$(f_j)_{j=1}^n\subset X(\mu)^{\frac{1}{p}}+X(\mu)$ and fix
$\varepsilon>0$. For each $j$, by \eqref{EQ: L1m-intnorm}, we can
take $\varphi_j\in\mathcal{S}\big(\Sigma_{X(\mu)}\big)$ such that
$|\varphi_j|\le1$ and
$$
\Vert
f_j\Vert_{L^1(m_T)}\le\Big(\frac{\varepsilon}{2^j}\Big)^{\frac{p}{q}}+\Vert
I_{m_T}(f_j\varphi_j)\Vert_E.
$$
Since $f_j\varphi_j\in  X(\mu)^{\frac{1}{p}}+X(\mu)$ as
$|f_j\varphi_j|\le|f_j|$, then
$I_{m_T}(f_j\varphi_j)=S(f_j\varphi_j)$. So, by using inequality
\eqref{EQ: t-inequality} and the $\frac{q}{p}$-concavity of $S$, we
have that
\begin{eqnarray*}
\sum_{j=1}^n\Vert f_j\Vert_{L^1(m_T)}^{\frac{q}{p}} & \le &
\sum_{j=1}^n\left(\Big(\frac{\varepsilon}{2^j}\Big)^{\frac{p}{q}}+\Vert
S(f_j\varphi_j)\Vert_E\right)^{\frac{q}{p}} \\
& \le &
\max\{1,2^{\frac{q}{p}-1}\}\left(\sum_{j=1}^n\frac{\varepsilon}{2^j}+
\sum_{j=1}^n\Vert S(f_j\varphi_j)\Vert_E^{\frac{q}{p}}\right) \\ &
\le & \max\{1,2^{\frac{q}{p}-1}\}
\left(\varepsilon+C^{\frac{q}{p}}\,\Big\Vert\Big(\sum_{j=1}^n
|f_j\varphi_j|^{\frac{q}{p}}\Big)^{\frac{p}{q}}\Big\Vert_{X(\mu)^{\frac{1}{p}}+X(\mu)}^{\frac{q}{p}}\right)
\\ &
\le & \max\{1,2^{\frac{q}{p}-1}\}
\left(\varepsilon+C^{\frac{q}{p}}\,\Big\Vert\Big(\sum_{j=1}^n
|f_j|^{\frac{q}{p}}\Big)^{\frac{p}{q}}\Big\Vert_{X(\mu)^{\frac{1}{p}}+X(\mu)}^{\frac{q}{p}}\right).
\end{eqnarray*}
Taking limit as $\varepsilon\to0$, we obtain
$$
\sum_{j=1}^n\Vert f_j\Vert_{L^1(m_T)}^{\frac{q}{p}}\le
C^{\frac{q}{p}}\,\max\{1,2^{\frac{q}{p}-1}\}
\Big\Vert\Big(\sum_{j=1}^n
|f_j|^{\frac{q}{p}}\Big)^{\frac{p}{q}}\Big\Vert_{X(\mu)^{\frac{1}{p}}+X(\mu)}^{\frac{q}{p}}
$$
and so
$$
\Big(\sum_{j=1}^n\Vert
f_j\Vert_{L^1(m_T)}^{\frac{q}{p}}\Big)^{\frac{p}{q}}\le
C\,\max\{1,2^{1-\frac{p}{q}}\} \Big\Vert\Big(\sum_{j=1}^n
|f_j|^{\frac{q}{p}}\Big)^{\frac{p}{q}}\Big\Vert_{X(\mu)^{\frac{1}{p}}+X(\mu)}.
$$
Hence, $[i]\colon X(\mu)^{\frac{1}{p}}+X(\mu)\to L^1(m_T)$ is
$\frac{q}{p}$-concave.

(c) $\Leftrightarrow$ (d) From Theorem \ref{THM:
quasiBfsXsubset(LpCapL1)mT}, we have that $[i]\colon
X(\mu)^{\frac{1}{p}}+X(\mu)\to L^1(m_T)$ is well defined if and only
if $[i]\colon X(\mu)\to L^p(m_T)\cap L^1(m_T)$ is well defined,
which is equivalent to $[i]\colon X(\mu)\to L^1(m_T)$ and $[i]\colon
X(\mu)\to L^p(m_T)$ well defined. By Lemma \ref{LEM:
q-concaveOperatorOnX+Y} we have that $[i]\colon
X(\mu)^{\frac{1}{p}}+X(\mu)\to L^1(m_T)$ is $\frac{q}{p}$-concave if
and only if $[i]\colon X(\mu)^{\frac{1}{p}}\to L^1(m_T)$ and
$[i]\colon X(\mu)\to L^1(m_T)$ are $\frac{q}{p}$-concave. On other
hand, it is straightforward to verify that $[i]\colon
X(\mu)^{\frac{1}{p}}\to L^1(m_T)$ is $\frac{q}{p}$-concave if and
only if $[i]\colon X(\mu)\to L^p(m_T)$ is $q$-concave.

(d) $\Leftrightarrow$ (e) follows from Proposition \ref{PROP:
q-concaveCore}.

(c) $\Rightarrow$ (a) Denote by $C$ the $\frac{q}{p}$-concavity
constant of $[i]\colon X(\mu)^{\frac{1}{p}}+X(\mu)\to L^1(m_T)$.
Consider $(f_j)_{j=1}^n\subset X(\mu)$ and note that $f_j\in
L^1(m_T)$ with $I_{m_T}(f_j)=T(f_j)$ for all $j$. Then,
\begin{eqnarray*}
\Big(\sum_{j=1}^n\Vert
T(f_j)\Vert_E^{\frac{q}{p}}\Big)^{\frac{p}{q}} & = &
\Big(\sum_{j=1}^n\Vert
I_{m_T}(f_j)\Vert_E^{\frac{q}{p}}\Big)^{\frac{p}{q}} \\ & \le &
\Big(\sum_{j=1}^n\Vert
f_j\Vert_{L^1(m_T)}^{\frac{q}{p}}\Big)^{\frac{p}{q}} \\ & \le &
C\,\Big\Vert\Big(\sum_{j=1}^n|f_j|^{\frac{q}{p}}\Big)^{\frac{p}{q}}\Big\Vert_{X(\mu)^{\frac{1}{p}}+X(\mu)}.
\end{eqnarray*}
\end{proof}

Note that $qL^p(m_T)=\big(\frac{q}{p}L^1(m_T)\big)^p$ (see
Proposition \ref{PROP: pPower-qX(mu)}). In particular, in the case
when $T$ is $(p,q)$-power-concave and $\chi_\Omega\in X(\mu)$ (so
$\chi_\Omega\in \frac{q}{p}L^1(m_T)$), from Remark \ref{REM:
XpResults}.(b) it follows that $\frac{q}{p}L^1(m_T)\cap
qL^p(m_T)=qL^p(m_T)$ if $p\ge1$ and $\frac{q}{p}L^1(m_T)\cap
qL^p(m_T)=\frac{q}{p}L^1(m_T)$ if $p<1$.

\begin{theorem}\label{THM: (p,q)PowerConcave-Factorization}
Suppose that $T$ is $(p,q)$-power-concave. Then, $T$ factors as
\begin{equation}\label{EQ: (p,q)PowerConcave-Factorization1}
\begin{split}
\xymatrix{
X(\mu) \ar[rr]^{T} \ar@{.>}[dr]_(.4){[i]} & & E\\
& \frac{q}{p}L^1(m_T)\cap qL^p(m_T) \ar@{.>}[ur]_(.6){I_{m_T}}}
\end{split}
\end{equation}
with $I_{m_T}$ being $(p,q)$-power-concave. Moreover, the
factorization is \emph{optimal} in the sense:
$$
\left.\begin{minipage}{6.2cm} \textnormal{\it If $Z(\xi)$ is a
$\sigma$-order continuous quasi-B.f.s.\ such that $\xi\ll\mu$ and}
\leqnomode
\begin{equation}\label{EQ: (p,q)PowerConcave-Factorization2}
\begin{split}
\xymatrix{
X(\mu) \ar[rr]^{T} \ar@{.>}[dr]_(.45){[i]} & & E\\
& Z(\xi) \ar@{.>}[ur]_(.5){S}}
\end{split}
\end{equation}
\textnormal{\it with $S$ being a $(p,q)$-power-concave linear
operator}
\end{minipage}\ \right\} \ \Longrightarrow \ \ \
\begin{minipage}{6cm}
\textnormal{\it $[i]\colon Z(\xi)\to \frac{q}{p}L^1(m_T)\cap
qL^p(m_T)$}
\\
\textnormal{\it is well defined and $S=I_{m_T}\circ[i]$.}
\end{minipage}
$$
\end{theorem}

\begin{proof}
The factorization \eqref{EQ: (p,q)PowerConcave-Factorization1}
follows from Theorem \ref{THM: Xsubset(q/pL1CapqLp)mT}. The space
$\frac{q}{p}L^1(m_T)\cap qL^p(m_T)$ is $\sigma$-order continuous as
noted before and satisfies the $\sigma$-property as $X(\mu)$ does.
Since $ {\textstyle I_{m_T}\colon \frac{q}{p}L^1(m_T)\cap
qL^p(m_T)\to E} $ is continuous (as $I_{m_T}\colon L^1(m_T)\to E$ is
so), we can apply Theorem \ref{THM: Xsubset(q/pL1CapqLp)mT} to see
that it is $(p,q)$-power-concave. Note that
$\Sigma_{X(\mu)}\subset\Sigma_{\frac{q}{p}L^1(m_T)\cap qL^p(m_T)}$
and $m_{I_{m_T}}(A)=I_{m_T}(\chi_A)=T(\chi_A)=m_T(A)$ for all
$A\in\Sigma_{X(\mu)}$. That is, $m_T$ is the restriction of
$m_{I_{m_T}}\colon\Sigma_{\frac{q}{p}L^1(m_T)\cap qL^p(m_T)}\to E$
to $\Sigma_{X(\mu)}$. From \cite[Lemma
3]{calabuig-delgado-sanchezperez}, it follows that
$L^1(m_{I_{m_T}})=L^1(m_T)$. Then,
$$
{\textstyle
[i]\colon \frac{q}{p}L^1(m_T)\cap qL^p(m_T)\to
\frac{q}{p}L^1(m_{I_{m_T}})\cap qL^p(m_{I_{m_T}})}
$$
is well defined as $\frac{q}{p}L^1(m_{I_{m_T}})\cap
qL^p(m_{I_{m_T}})=\frac{q}{p}L^1(m_T)\cap
qL^p(m_T)$.

Let $Z(\xi)$ satisfy \eqref{EQ: (p,q)PowerConcave-Factorization2}.
In particular, $Z(\xi)$ has the $\sigma$-property. From Theorem
\ref{THM: Xsubset(q/pL1CapqLp)mT} applied to the operator $S\colon
Z(\xi)\to E$, we have that $[i]\colon Z(\xi)\to
\frac{q}{p}L^1(m_S)\cap qL^p(m_S)$ is well defined and
$S=I_{m_S}\circ[i]$. Since $\Sigma_{X(\mu)}\subset\Sigma_{Z(\xi)}$
and $m_S(A)=S(\chi_A)=T(\chi_A)=m_T(A)$ for all
$A\in\Sigma_{X(\mu)}$ (i.e.\ $m_T$ is the restriction of
$m_S\colon\Sigma_{Z(\xi)}\to E$ to $\Sigma_{X(\mu)}$), from
\cite[Lemma 3]{calabuig-delgado-sanchezperez}, it follows that
$L^1(m_S)=L^1(m_T)$ and $I_{m_S}=I_{m_T}$. Therefore,
$$
{\textstyle[i]\colon Z(\xi)\to \frac{q}{p}L^1(m_S)\cap
qL^p(m_S)=\frac{q}{p}L^1(m_T)\cap qL^p(m_T)}
$$
is well defined and
$S=I_{m_S}\circ[i]=I_{m_T}\circ[i]$.
\end{proof}

We can rewrite Theorem \ref{THM: (p,q)PowerConcave-Factorization} in
terms of optimal domains.

\begin{corollary}
Suppose that $T$ is $(p,q)$-power-concave. Then
$\frac{q}{p}L^1(m_T)\cap qL^p(m_T)$ is the largest $\sigma$-order
continuous quasi-B.f.s.\ to which $T$ can be extended as a
$(p,q)$-power-concave operator still with values in $E$. Moreover,
the extension of $T$ to the space $\frac{q}{p}L^1(m_T)\cap
qL^p(m_T)$ is given by the integration operator $I_{m_T}$.
\end{corollary}

Recalling that the $(1,q)$-power-concave operators coincide with the
$q$-concave operators, we obtain our main result.

\begin{corollary} \label{COR: qConcaveOptimalDomain}
Suppose that $T$ is $q$-concave. Then $qL^1(m_T)$ is the largest
$\sigma$-order continuous quasi-B.f.s.\ to which $T$ can be extended
as a $q$-concave operator still with values in $E$. Moreover, the
extension of $T$ to $qL^1(m_T)$ is given by the integration operator
$I_{m_T}$.
\end{corollary}

Let us give now a direct application related to the Maurey-Rosenthal
factorization of $q$-concave operators defined on a $q$-convex
quasi-B.f.s. In the case when $T$ is $q$-concave, by Corollary
\ref{COR: qConcaveOptimalDomain}, the integration operator $I_{m_T}$
extends $T$ to the space $qL^1(m_T)$. Note that the map $[i]\colon
X(\mu)\to qL^1(m_T)$ is $q$-concave as it is continuous and
$qL^1(m_T)$ is $q$-concave. From a variant of the Maurey-Rosenthal
theorem proved in \cite[Corollary 5]{defant}, under some extra
conditions, if $X(\mu)$ is $q$-convex then $[i]\colon X(\mu)\to
qL^1(m_T)$ factors through the space $L^q(\mu)$. So, we obtain the
following improvement of the usual factorization of $q$-concave
operators on $q$-convex quasi-B.f.s.'.

\begin{corollary}\label{COR: MaureyRosenthalFactorization}
Let $1\le q<\infty$. Assume that $\mu$ is $\sigma$-finite and that
$X(\mu)$ is $q$-convex and has the $\sigma$-Fatou property. If $T$
is $q$-concave then it can be factored as
$$
\xymatrix{
X(\mu) \ar[rr]^{T} \ar@{.>}[d]_{M_g} & &   E\\
L^q(\mu)  \ar@{.>}[rr]^{M_{g^{-1}}}&  & qL^1(m_T)
\ar@{.>}[u]_{I_{m_T}}}
$$
for positive multiplication operators $M_g$ and $M_{g^{-1}}$. The
converse is also true.
\end{corollary}


\section{Vector measure representation of $q$-concave Banach lattices}
\label{SEC: qConcaveBanachLattices}

In this last section we look for a characterization of the class of
Banach lattices which are $p$-convex and $q$-concave in terms of
spaces of integrable functions with respect to a vector measure. For
$1<p$, it is known that order continuous $p$-convex Banach lattices
can be order isometrically represented as spaces $L^p$ of a vector
measure defined on a $\delta$-ring (see \cite[Theorem
10]{calabuig-juan-sanchezperez}). We will see that the addition of
the $q$-concavity property to the represented Banach lattice
translates to adding some concavity property to the corresponding
integration map.

First let us show two results concerning concavity for the
integration operator of a vector measure which will be needed later.

Let $m\colon\mathcal{R}\to E$ be a vector measure defined on a
$\delta$--ring $\mathcal{R}$ of subsets of $\Omega$ and with values
in a Banach space $E$.

\begin{proposition}\label{PROP: Im-qConcave}
The integration operator $I_m\colon L^1(m)\to E$ is $q$-concave if
and only if $L^1(m)$ is $q$-concave.
\end{proposition}

\begin{proof}
Suppose that $I_m\colon L^1(m)\to E$ is $q$-concave and denote by
$C$ its $q$-concavity constant. Take $(f_j)_{j=1}^n\subset L^1(m)$
and $(\varphi_j)_{j=1}^n\subset \mathcal{S}(\mathcal{R})$ with
$|\varphi_j|\le1$ for all $j$. Since $(f_j\varphi_j)_{j=1}^n\subset
L^1(m)$, as $|f_j\varphi_j|\le|f_j|$ for all $j$, we have that
$$
\Big(\sum_{j=1}^n\Vert
I_m(f_j\varphi_j)\Vert_E^q\Big)^{\frac{1}{q}}\le
C\,\Big\Vert\Big(\sum_{j=1}^n|f_j\varphi_j|^q\Big)^{\frac{1}{q}}\Big\Vert_{L^1(m)}\le
C\,\Big\Vert\Big(\sum_{j=1}^n|f_j|^q\Big)^{\frac{1}{q}}\Big\Vert_{L^1(m)}.
$$
Taking supremum for each $j=1,..,n$ over all $\varphi_j\in
\mathcal{S}(\mathcal{R})$ with $|\varphi_j|\le1$, from \eqref{EQ:
L1m-intnorm}, if follows that
$$
\Big(\sum_{j=1}^n\Vert f_j\Vert_{L^1(m)}^q\Big)^{\frac{1}{q}}\le
C\,\Big\Vert\Big(\sum_{j=1}^n|f_j|^q\Big)^{\frac{1}{q}}\Big\Vert_{L^1(m)}.
$$
The converse is obvious as $I_m$ is continuous.
\end{proof}

Direct useful consequences can be deduced of the fact that the
integration map $I_m:L^1(m) \to E$ is $q$-concave. Assume that $m$
is defined on a $\sigma$-algebra and note that $q$-concavity for
$q\ge1$ always implies $(q,1)$-concavity (see the definition for
instance in \cite[p.\,61]{okada-ricker-sanchezperez2}). Thus, by
\cite[Proposition 7.9]{okada-ricker-sanchezperez2}, if $I_m$ is
$q$-concave for $q\ge1$ then it is \emph{weakly completely
continuous} (i.e.\ it maps weak Cauchy sequences into weakly
convergent sequences). Moreover, this implies that $L^1(m)$
coincides with the space $L_w^1(m)$ and so it has the $\sigma$-Fatou
property.

In the case when $\chi_\Omega\in L^1(m)$ (for instance if $m$ is
defined on a $\sigma$-algebra), we obtain a further result regarding
$(p,q)$-power-concave operators.

\begin{proposition}\label{PROP: Lp(m)-qConcave}
Suppose that $\chi_\Omega\in L^1(m)$ and $p\ge1$. The integration
operator $I_m\colon L^p(m)\to E$ is $(p,q)$-power-concave if and
only if $L^p(m)$ is $q$-concave.
\end{proposition}

\begin{proof}
First note that under the hypothesis it follows that $L^p(m)$ has
the $\sigma$-property (in fact $\chi_\Omega\in L^p(m)$) and
$L^p(m)\subset L^1(m)$. So, $I_m\colon L^p(m)\to E$ is well defined
and continuous.

Suppose that $I_m\colon L^p(m)\to E$ is $(p,q)$-power-concave. From
Theorem \ref{THM: Xsubset(q/pL1CapqLp)mT}, we have that $[i]\colon
L^p(m)\to\frac{q}{p}L^1(m_{I_m})\cap qL^p(m_{I_m})$ is well defined.
Note that $(\mathcal{R}^{loc})_{L^p(m)}=\mathcal{R}^{loc}$ and so
$m_{I_m}$ coincides with $m_{\chi_\Omega}$ (see Preliminaries).
Then, $L^1(m_{I_m})=L^1(m)$ and so
$L^p(m)\subset\frac{q}{p}L^1(m)\cap qL^p(m)\subset qL^p(m)$. Hence,
$L^p(m)$ is $q$-concave as $L^p(m)= qL^p(m)$.

Suppose now that $L^p(m)$ is $q$-concave. Then, it is direct to
check that $L^1(m)$ is $\frac{q}{p}$-concave. Since $L^p(m)\subset
L^1(m)$, the integration operator $I_m\colon L^1(m)\to E$ is
continuous and $\big(L^p(m)\big)^{\frac{1}{p}}+L^p(m)=L^1(m)$, it
follows that $I_m\colon L^p(m)\to E$ satisfies the inequality of the
definition of $(p,q)$-power-concave operator.
\end{proof}

Let us go now to the representation of $q$-concave Banach lattices
as spaces of integrable functions. We begin by considering B.f.s.'.

\begin{proposition}\label{PROP: Bfs-Representation}
Let $p,q\in(0,\infty)$ and let $Z(\xi)$ be a $q$-concave B.f.s.\
which is also $p$-convex in the case when $p>1$. Then, $Z(\xi)$
coincides with the space $L^p(m)$ of a Banach space valued vector
measure $m\colon\mathcal{R}\to E$ defined on a $\delta$-ring whose
integration operator $I_m\colon L^1(m)\to E$ is
$\frac{q}{p}$-concave. Moreover, if $\chi_\Omega\in Z(\xi)$, the
vector measure $m$ is defined on a $\sigma$-algebra.
\end{proposition}

\begin{proof}
Note that if $p\le1$ then $Z(\xi)^{\frac{1}{p}}$ is a B.f.s.\ (see
Remark \ref{REM: XpResults}.(d)). In the case when $p>1$, renorming
$Z(\xi)$ if it is necessary, we can assume that the $p$-convexity
constant of $Z(\xi)$ is equal to $1$ (see \cite[Proposition
1.d.8]{lindenstrauss-tzafriri}), and so $Z(\xi)^{\frac{1}{p}}$ is a
B.f.s. (see Remark \ref{REM: XpResults}.(e)). Consider the
$\delta$-ring $\Sigma_{Z(\xi)}=\big\{A\in\Sigma:\,\chi_A\in
Z(\xi)\big\}$ and the finitely additive set function
$m\colon\Sigma_{Z(\xi)}\to Z(\xi)^{\frac{1}{p}}$ given by
$m(A)=\chi_A$. Since $Z(\xi)^{\frac{1}{p}}$ is $\sigma$-order
continuous, as $Z(\xi)$ is so by Proposition \ref{PROP:
q-concaveImpliesSigmaoc}, it follows that $m$ is a vector measure.
Let us see that $L^1(m)=Z(\xi)^{\frac{1}{p}}$ with equal norms and
so we will have that $Z(\xi)$ coincides with $L^p(m)$. For
$\varphi\in\mathcal{S}(\Sigma_{Z(\xi)})$ we have that $\varphi\in
Z(\xi)^{\frac{1}{p}}$ and $I_m(\varphi)=\varphi$. Moreover, since
$m$ is positive,
\begin{equation}\label{EQ: Bfs-Representation}
\Vert\varphi\Vert_{L^1(m)}=\Vert
I_m(|\varphi|)\Vert_{Z(\xi)^{\frac{1}{p}}}=\Vert
\varphi\Vert_{Z(\xi)^{\frac{1}{p}}}.
\end{equation}
In particular, by taking $\varphi=\chi_A$, we obtain that $\Vert
m\Vert$ is equivalent to $\xi$. Given $f\in L^1(m)$, since
$\mathcal{S}(\Sigma_{Z(\xi)})$ is dense in $L^1(m)$, we can take
$(\varphi_n)\subset\mathcal{S}(\Sigma_{Z(\xi)})$ such that
$\varphi_n\to f $ in $L^1(m)$ and $m$-a.e. From \eqref{EQ:
Bfs-Representation}, we have that $(\varphi_n)$ is a Cauchy sequence
in $Z(\xi)^{\frac{1}{p}}$ and so there is $h\in
Z(\xi)^{\frac{1}{p}}$ such that $\varphi_n\to h$ in
$Z(\xi)^{\frac{1}{p}}$. Taking a subsequence $\varphi_{n_j}\to h$
$\xi$-a.e.\ we see that $f=h\in Z(\xi)^{\frac{1}{p}}$ and
$$
\Vert f\Vert_{Z(\xi)^{\frac{1}{p}}}=\lim\Vert
\varphi_n\Vert_{Z(\xi)^{\frac{1}{p}}}=\lim\Vert
\varphi_n\Vert_{L^1(m)}=\Vert f\Vert_{L^1(m)}.
$$
Let now $f\in Z(\xi)^{\frac{1}{p}}$ and take
$(\varphi_n)\subset\mathcal{S}(\Sigma)$ such that
$0\le\varphi_n\uparrow|f|$. For any $n$, writing
$\varphi_n=\sum_{j=1}^m\alpha_j\chi_{A_j}$ with $(A_j)_{j=1}^m$
being pairwise disjoint and $\alpha_j>0$ for all $j$, we see that
$\chi_{A_j}\le\alpha_j^{-1/p}|f|^{1/p}$ and so
$\varphi_n\in\mathcal{S}(\Sigma_{Z(\xi)})$. On other hand, since
$Z(\xi)^{\frac{1}{p}}$ is $\sigma$-order continuous, we have that
$\varphi_n\to f$ in $Z(\xi)^{\frac{1}{p}}$. From \eqref{EQ:
Bfs-Representation}, we have that $(\varphi_n)$ is a Cauchy sequence
in $L^1(m)$ and so there is $h\in L^1(m)$ such that $\varphi_n\to h$
in $L^1(m)$. Taking a subsequence $\varphi_{n_j}\to h$ $m$-a.e.\ we
see that $f=h\in L^1(m)$.

Hence, $L^1(m)=Z(\xi)^{\frac{1}{p}}$ with equal norms and, since
$Z(\xi)$ is $q$-concave, it follows that $L^1(m)$ is
$\frac{q}{p}$-concave. From Proposition \ref{PROP: Im-qConcave}, the
integration operator $I_m\colon L^1(m)\to E$ is
$\frac{q}{p}$-concave.

Note that if $\chi_\Omega\in Z(\xi)$, then $\Sigma_{Z(\xi)}=\Sigma$
and so $m$ is defined on a $\sigma$-algebra.
\end{proof}

For the final result we need some concepts related to Banach
lattices. The definitions of $p$-convexity, $q$-concavity and
$\sigma$-order continuity for Banach lattices are the same that for
B.f.s.'. A Banach lattice $F$ is said to be \emph{order continuous}
if for every downwards directed system $(x_\tau)\subset F$ with
$x_\tau\downarrow0$ it follows that $\Vert x_\tau\Vert_F\downarrow0$
and is said to be \emph{$\sigma$-complete} if every order bounded
sequence in $F$ has a supremum. A Banach lattice which is
$\sigma$-order continuous and $\sigma$-complete at the same time is
order continuous, see \cite[Proposition
1.a.8]{lindenstrauss-tzafriri}. A \emph{weak unit} of a Banach
lattice $F$ is an element $0\le e\in F$ such that $\inf\{x,e\}=0$
implies $x=0$. An operator $T\colon F_1\to F_2$ between Banach
lattices is said to be an \emph{order isometry} if it is linear, one
to one, onto, $\Vert Tx\Vert_{F_2}=\Vert x\Vert_{F_1}$ for all $x\in
F_1$ and $T(\inf\{x,y\})=\inf\{Tx,Ty\}$ for all $x,y\in F_1$. In
particular, an order isometry is a positive operator. So, by using
\cite[Proposition 1.d.9]{lindenstrauss-tzafriri}, it is direct to
check that every order isometry preserves $p$-convexity and
$q$-concavity whenever $p,q\ge1$.

\begin{theorem}\label{THM: BanachLattice-Representation}
Let $p,q\in[1,\infty)$ and let $F$ be a Banach lattice. The
following statements are equivalent:
\begin{itemize}\setlength{\leftskip}{-3ex}\setlength{\itemsep}{.5ex}
\item[(a)] $F$ is $q$-concave and $p$-convex.

\item[(b)] $F$ is order isometric to a space $L^p(m)$ of
a Banach space valued vector measure $m\colon\mathcal{R}\to E$
defined on a $\delta$-ring whose integration operator $I_m\colon
L^1(m)\to E$ is $\frac{q}{p}$-concave.
\end{itemize}
Moreover, (a) holds with $F$ having a weak unit if and only if (b)
holds with $m$ defined on a $\sigma$-algebra. In this last case
$I_m\colon L^p(m)\to E$ is $(p,q)$-power-concave.
\end{theorem}

\begin{proof}
(a) $\Rightarrow$ (b) Since $F$ is $q$-concave, it satisfies a lower
$q$-estimate (see \cite[Definition 1.f.4]{lindenstrauss-tzafriri})
and then it is $\sigma$-complete and $\sigma$-order continuous (see
the proof of \cite[Proposition 1.f.5]{lindenstrauss-tzafriri}). So,
$F$ is order continuous. From \cite[Theorem 5]{delgado-juan} we have
that $F$ is order isometric to a space $L^1(\nu)$ of a Banach space
valued vector measure $\nu$ defined on a $\delta$-ring. Then,
$L^1(\nu)$ is a B.f.s.\ satisfying the conditions of Proposition
\ref{PROP: Bfs-Representation} and so $L^1(\nu)=L^p(m)$ with
$m\colon\mathcal{R}\to E$ being a vector measure defined on a
$\delta$-ring $\mathcal{R}$ and with values in a Banach space $E$,
whose integration operator $I_m\colon L^1(m)\to E$ is
$\frac{q}{p}$-concave.

(b) $\Rightarrow$ (a) Since $L^p(m)$ is $p$-convex (Remark \ref{REM:XpResults}.(c)) and $q$-concave (as $L^1(m)$ is
$\frac{q}{p}$-concave by Proposition \ref{PROP: Im-qConcave}), $F$
also is.

Now suppose that (a) holds with $F$ having a weak unit. From
\cite[Theorem 8]{curbera} we have that $F$ is order isometric to a
space $L^1(\nu)$ of a Banach space valued vector measure $\nu$
defined on a $\sigma$-algebra. Since $\chi_\Omega\in L^1(\nu)$, from
Proposition \ref{PROP: Bfs-Representation} we have that (b) holds
with $m$ defined on a $\sigma$-algebra.

Conversely, if (b) holds with $m$ defined on a $\sigma$-algebra then
$\chi_\Omega\in L^p(m)$ (as $\chi_\Omega\in L^1(m)$). So, the image
of $\chi_\Omega$ by the order isometry is a weak unit in $F$.
Moreover, from Proposition \ref{PROP: Lp(m)-qConcave} it follows
that $I_m\colon L^p(m)\to E$ is $(p,q)$-power-concave.
\end{proof}

In particular, from Theorem \ref{THM: BanachLattice-Representation}
we obtain that a Banach lattice is $q$-concave (with $q\ge1$) if and
only if it is order isometric to a space $L^1(m)$ of a vector
measure $m$ with a $q$-concave integration operator.


\end{document}